\documentclass[a4paper,12pt]{amsart}
\usepackage{fullpage}
\usepackage{amssymb}
\usepackage{amsthm}
\usepackage{mathtools}
\usepackage{graphicx}
\usepackage{tikz}
\usepackage{tikz-cd}

\usepackage[titletoc]{appendix}

\newcounter{counter2}
\numberwithin{equation}{section}
\numberwithin{counter2}{section}

\newtheorem{thm}[counter2]{Theorem}
\newtheorem{prop}[counter2]{Proposition}
\newtheorem{lemma}[counter2]{Lemma}
\newtheorem{cor}[counter2]{Corollary}
\newtheorem{defn}[counter2]{Definition}
\newtheorem{rem}[counter2]{Remark}
\newtheorem{exmp}[counter2]{Example}

\newcommand{\im}{\operatorname{im}}

\newcommand{\End}{\operatorname{End}}

\DeclareSymbolFont{script}{U}{eus}{m}{n}
\DeclareMathSymbol{\Wedge}{0}{script}{"5E}

\newcounter{mnotecount}[section]

\renewcommand{\themnotecount}{\thesection.\arabic{mnotecount}}

\newcommand{\mnote}[1]%{}%
{\protect{\stepcounter{mnotecount}}$^{\mbox{\footnotesize
$%\!\!\!\!\!\!\,
\bullet$\themnotecount}}$ \marginpar{%\color{red}%
\raggedright\tiny\em
$\!\!\!\!\!\!\,\bullet$\themnotecount: #1} }

\newcommand{\X}{{X}}

\def\be{\begin{equation}}

\def\dim{\mbox{dim}}

\def\dim{\mbox{dim}}

\def\ee{\end{equation}}

\def\bea{\begin{eqnarray}}
\def\eea{\end{eqnarray}}

\newcommand{\nocontentsline}[3]{}
\newcommand{\tocless}[2]{\bgroup\let\addcontentsline=\nocontentsline#1{#2}\egroup}

\title{Conformally Einstein anti-self-dual spaces \\ and their generalisations}
\author{Timothy Moy$^{1,2}$}
\thanks{${}^{1}$Department of Applied Mathematics and Theoretical Physics,  
University of Cambridge}
\thanks{${}^2$London Mathematical Society Early Career Research Fellow,  School of Mathematical and Physical Sciences,  University of Sheffield}

\begin{document}

\begin{abstract} 
We study the question of existence of Einstein scales and compatible metrics in Grassmannian geometry.  As a special case,  this includes the conformal-to-Einstein problem for anti-self-dual (ASD) conformal structures.  \par 
In the ASD setting,  given a genericity condition algebraic in the Weyl curvature,  we obtain necessary and sufficient conditions for the existence of a local Einstein scale that refine those previously appearing in the literature.  We also prove that in Riemannian signature,  a non-Ricci-flat ASD-K\"ahler metric is locally conformally Einstein if and only if it has a holomorphic Killing vector field with anti-self-dual derivative proportional to the Ricci spinor.  \par  Next,  we consider some generalisations of the ASD conformal-to-Einstein problem for $(2n,2)$-Grassmannian structures.  When $n > 1$,  given an analogous genericity condition,  we obtain necessary and sufficient conditions for the existence of a local Einstein scale.  Using tractor calculus,  we obtain algebraic obstructions to the existence of compatible metrics and show that the unique submaximally symmetric model of the geometry has the submaximal number of linearly independent compatible metrics.

%
%Given an anti-self-dual (ASD) conformal structure $(X,[g])$ on a four-manifold $X$,  we review the conformal-to-Einstein problem of determining when the conformal class (locally) admits an Einstein scale.  We observe that known arguments in the non-self-dual,  algebraically generic case may be readily adapted and we write down necessary and sufficient conditions in terms of spinor invariants.  These refine conditions previously appearing in the literature.  Next,  we consider the problem for a number of examples and special classes of four-dimensional ASD metrics.  Included is a proof that in Riemannian signature,  an ASD-K\"ahler metric is conformally Einstein if and only if it has a holomorphic Killing vector with self-dual derivative proportional to the Ricci spinor.  We then study a generalisation of the ASD conformal-to-Einstein problem.  Rather than consider conformal structures in higher dimensions,  we extend our analysis to the setting of $(2n,2)$-almost-Grassmannian structures,  $n > 1$,  and the problem of finding metrics (all of which are Einstein) that are compatible in the appropriate sense.  
\end{abstract}

\maketitle
 \setcounter{tocdepth}{1}
\tableofcontents
\section{Introduction} 
In this article we seek conditions for the existence of Einstein metrics compatible with a torsion-free Grassmannian structure.  In four dimensions,  this is simply the conformal-to-Einstein problem for anti-self-dual (ASD) metrics.  Since this special case is of interest in general relativity and mathematical physics,  we study it in detail,  and the material up to and including \S \ref{ASDK} can be read entirely with a view to the four-dimensional case.  
\subsection*{Einstein scales in anti-self-dual conformal geometry}
The problem of determining when a conformal class of metrics contains an Einstein representative,  the \textit{conformal-to-Einstein} problem,  was considered in the early years of general relativity \cite{Br24}.   In the four-dimensional Lorentzian setting,  sets of tensor invariants of which the vanishing is necessary and sufficient for the existence of a locally compatible Einstein metric were obtained in \cite{S63} and \cite{KNT85},  given an algebraic non-degeneracy assumption on the Weyl curvature.  There is now a large corpus of literature on the problem both in four and higher dimensions \cite{GN06,KNN03,L01,L06} and related global problems \cite{D83,CLW08},  to list just some examples.   \par 
Staying in four dimensions,  in the Riemannian,  split-signature,  and holomorphic settings,  one may consider conformal structures for which the Weyl tensor (thought of as an endomorphism on two-forms) is anti-self-dual.  This class of metrics has been historically important,  as it is the setting of curved twistor theory \cite{P76}.  Since the endomorphism defined by the Weyl tensor necessarily has a kernel,  these metrics fall outside the classes considered in most treatments of the conformal-to-Einstein problem.  The problem has been considered in the ASD setting in \cite{BE90} and (for the case of Ricci-flat scales) in \cite{DT14}.  The treatment in \cite{GN06} for ``weakly generic" pseudo-Riemannian metrics in arbitrary dimension specialises to give necessary and sufficient conditions for a generic ASD metric to be conformally Einstein.  There,  necessary and sufficient conditions are stated as the vanishing of a section of a bundle of rank fifteen.    \par 
In \S \ref{KNT_repeat},  we apply similar methods to \cite{S63,  KNT85} to study Einstein scales for ASD metrics.  We use them to refine some of the results in \cite{DT14,  GN06} in the generic setting.  In particular,  using a result of \cite{Calderbank},  we can show that the necessary condition found in \cite{BE90} is in fact sufficient.  After picking a basis it consists of just four linearly independent equations which are third (differential) order in the metric components.  \par
In \S \ref{ASDK},  we consider the problem for the subclass of ASD-K\"ahler (scalar-flat K\"ahler) metrics.  We prove that in Riemannian signature,  a non-Ricci-flat ASD-K\"ahler metric is conformally Einstein if and only if it has a holomorphic Killing vector with anti-self-dual derivative proportional to the Ricci spinor.  This completes the local characterisation of four-dimensional Riemannian K\"ahler metrics that are conformally Einstein,  as the non-ASD case is treated in \cite{D83}.  \par We compute,  using the Maple \texttt{DifferentialGeometry} package,  some interesting examples: ASD Pleba\'nski-Demia\'nski metrics \cite{PD76,  A23} and an example arising from three-dimensional Einstein-Weyl geometry on the space Nil.   
\subsection*{Einstein scales and compatible metrics in Grassmannian geometry} 
Central to twistor theory is the interpretation of a conformal structure $(X,[g])$ in four dimensions as a ``curved analogue" of the Grassmannian $F_2(\mathbb{C}^4)$ of two-planes,  and also the fact that a version of the \textit{twistor correspondence},  underlain by the double fibration 
\begin{equation}
\begin{tikzcd}
                  & {F_{1,2}(\mathbb{C}^4)} \arrow[rd] \arrow[ld] &                   \\
\mathbb{CP}^3 &                                               & F_2(\mathbb{C}^4)
\end{tikzcd}
\end{equation}
persists,  subject to an integrability condition equivalent to the anti-self-duality of the Weyl curvature.   
It is arguably more natural,  from the point of view of twistor theory,  to generalise four-dimensional conformal geometry not by considering higher-dimensional conformal structures,  but by considering spaces modelled on $F_2(\mathbb{C}^{2n+2})$ in the sense of \textit{parabolic geometry} \cite{CS09}.  \par 
From this perspective,  the structure one should consider on a $4n$-dimensional complex manifold $X$ is a local factorisation of the holomorphic tangent bundle
\begin{align}
TX \cong {S} \otimes {S}' 
\end{align}
where $S$ and $S'$ are holomorphic vector bundles of rank $2n$ and rank $2$ respectively.  As we will recall,  the invariant curvature of such a structure behaves similarly to the invariant curvature of a four-dimensional conformal structure: Firstly,  it can be decomposed into two irreducible components.  Secondly,  the vanishing of these ``right and left" components can be interpreted as the integrability conditions for the existence of maximal families of submanifolds with tangent spaces generalising the notion of self-dual and anti-self-dual planes,  respectively.  \par
The geometric structures we will study have been called \textit{Segr\'e},  \textit{generalised conformal} \cite{G87, G88},  \textit{paraconformal} \cite{BE91},  \textit{$(2n,2)$-almost-Grassmannian}\index{almost-Grassmannian} \cite{CS09,HSSS12},  or \textit{Grassmannian spinor} \cite{M97}.  There is also the related notion of \textit{pluricomplex} structures \cite{BS13},  which arise as particular examples of such structures equipped with an anti-holomorphic involution realising the underlying manifold as the analytic continuation of a real-analytic manifold.  \par 
%The structures we will consider in this article always satisfy the integrability condition generalising anti-self-duality (the ``right-flat" condition).  
We call them almost-Grassmannian,  leaving the type--$(2n,2)$ implicit,  and later specialise to the torsion-free case,  and use the acronym TFG for torsion-free Grassmannian.   \par  
In \S \ref{metrisability} we consider an overdetermined system of invariantly defined partial differential equations (PDE) on an almost-Grassmannian structure.  In the torsion-free case,  the non-degenerate solutions of this system correspond to metrics compatible with the Grassmannian geometry in the sense of \cite{BE91}.  When $n > 1$,  these metrics are all Einstein and are holomorphic analogues of quaternion-K\"ahler metrics.   By prolongation of the system one may show that solutions correspond to the parallel sections of the Cartan connection $D$ on a \textit{tractor bundle}.  This allows for an analysis of obstructions to the existence of solutions,  and we construct some (\S \ref{tractor_obstructions}).   In \S \ref{submax} we consider the submaximally symmetric Grassmannian structure obtained by Kruglikov and The \cite{KT}.  We show that this example has the submaximally allowed number of linearly independent compatible metrics.  
\par Finally,  in \S \ref{AGEinstein} we consider the easier problem of finding compatible connections for the geometry which satisfy an analogue of the Einstein condition.  In fact,  as is noted in \cite{BE91},  such an \textit{Einstein scale} generically arises from a compatible metric.  On a torsion-free Grassmannian structure,  there is an analogous identity to the vanishing of the Bach tensor in ASD geometry.  This identity allows us to adapt the spinorial arguments of \cite{KNT85} to the Grassmannian setting,  and the necessary and sufficient conditions for the existence of an Einstein scale (given an algebraic genericity condition) resemble those in four dimensions. 
Before all this,  in \S \ref{spinors_and_spacetime},  we review $(2n,2)$-Grassmannian geometry.  All formulae are valid for anti-self-dual conformal structures.  \par 
\subsection*{Acknowledgements}
This article began as a chapter of the author's PhD thesis \cite{Moy26} and therefore owes a great deal to the guidance and supervision of Maciej Dunajski. 
The author would like to thank Bernardo Araneda,  Alex Colling,  Michael Eastwood,  Rod Gover,  Zhangwen Guo,  Derek Harland and Sean Seet for helpful conversations. The author would also like to acknowledge generous hospitality provided to them during their PhD:  
Much of the research for this article was carried out during the GRIEG meeting at the Laboratoire de Mathématiques d'Orsay in March 2024,  and 
when the author was hosted by the Ward family in January 2025.  The research was supported by the Simons Foundation
grant  SFI-MPS-T-Institutes-00010825,  and by the State Treasury funds as part of a
task commissioned by the Minister of Science and Higher Education under the project Organization of the Simons Semesters at the Banach Center - New Energies in 2026-2028 (MNiSW/2025/DAP/491). 
\section{Calculus on (almost-)Grassmannian structures}\label{spinors_and_spacetime}\index{spinor!calculus}
\begin{defn}[Almost-Grassmannian structure]
A $(2n,2)$-almost-Grassmannian structure\footnote{Our definition is slightly different to \textup{\cite{CS09}}.  Like \textup{\cite{M97}} we do not require the additional data of an isomorphism of the top exterior powers of $S$ and $S'$.  In  \textup{\cite{CS09}} the  extra data corresponds to a lift of the structure to $S(GL(2n,\mathbb{C}) \times GL(2,\mathbb{C}))$ from the subgroup of $GL(4n,\mathbb{C})$ generated by letting $GL(2n,\mathbb{C})$ and $GL(2,\mathbb{C})$ act independently on each factor of the decomposition $\mathbb{C}^{4n} \cong \mathbb{C}^{2n} \otimes \mathbb{C}^2$.  As we will see directly,  the assumption of this extra data is harmless locally.} on a complex $4n$-dimensional manifold $X$ is an isomorphism 
\begin{align}\label{paraconformal}
TX \cong S \otimes S',
\end{align}
where $S,  S'$ are vector bundles of rank $2n$ and $2$ respectively.  
\end{defn}

This generalises the canonical decomposition of the tangent bundle of the Grassmannian $F_2(\mathbb{C}^{2n+2})$ as a tensor product.  In this flat model,  the two factors are defined by the exact sequence
\begin{equation}
\begin{tikzcd}
0 \arrow[r] & S'^* \arrow[r] & T \arrow[r] & \underbrace{T/S'^*}_{:=S} \arrow[r] & 0
\end{tikzcd}
\end{equation}
where $T\to F_2(\mathbb{C}^{2n+2})$ is the trivial bundle with fibre $\mathbb{C}^{2n+2}$ and $S'^*$ is defined to be the tautological bundle.  Then,  the fibre of $TF_2(\mathbb{C}^{2n+2})$ at a point $x$ is identified with $\text{Hom}(S'^*_x,S_x)$ in the natural way.  
If $n = 1$,  sections of $S$ and $S'$ are called spinors.  For general $n$,  we will continue to refer to sections of $S'$ this way.  \index{spinor} \par 
Obviously,  there are global topological obstructions to the existence of a global decomposition (\ref{paraconformal}).  We assume we are working over a contractible neighbourhood without comment.  \par 
 \par Given an identification (\ref{paraconformal}) on $X$,  for each $x \in X$ there is a distinguished cone of vectors in $T_x\X$ of the form
\be
\label{null}
v=p\otimes q, \quad\mbox{where}\quad p\in S_x, q\in S'_x.
\ee
Fixing $q$ while varying $p$  the vectors (\ref{null}) span a $2n$--dimensional subspace of $T_x\X$ called an \textit{$\alpha$--plane.}\index{alpha-plane@$\alpha$-plane}  
Similarly,  for each fixed $p\in S_x$ there is a $2$--dimensional \textit{$\beta$--plane}\index{beta-plane@$\beta$-plane} obtained by varying $q$
in (\ref{null}).  $\alpha$--planes and $\beta$--planes generalise the classification of null planes into self-dual or anti-self-dual planes for a holomorphic conformal structure in four dimensions.  
A $2n$--dimensional submanifold of $X$ is called an $\alpha$--surface\index{alpha-surface@$\alpha$-surface} if its tangent planes are $\alpha$--planes.  A distribution consisting of $\alpha$-planes is called an $\alpha$-distribution.  $\beta$-surfaces\index{beta-surface@$\beta$-surface} and distributions are defined analogously.  \par 
In the terminology of \cite{BE91},  the structure is called \textit{right--flat}\index{right-flatness} if there is an $\alpha$--surface tangent to each $\alpha$--plane.  This is equivalent to the anti-self-duality condition if $n=1$.  Given the right--flat condition,  the twistor space may be defined as the $2n+1$-dimensional space of $\alpha$-surfaces\footnote{As a word of warning,  it is also common in literature for the projective primed spin bundle $\mathbb{P}S'$ to be referred to as the twistor space.}.  \par 
There is also the notion of \textit{algebraically compatible metrics}\index{compatible metrics!algebraically compatible}.  
\begin{align}
g = \eta \otimes \eta'
\end{align}
for $\eta,  \eta'$ non-degenerate sections of $\wedge^2S^*$ and $\wedge^2{S'}^*$ respectively.  If $n=1$ this is the conformal class defined by decreeing that the decomposable vectors under the isomorphism (\ref{paraconformal}) span the null cone.  For $n > 1$,  the decomposable vectors are a strict subset of the null vectors of such a metric.  
 \par 
 There is an additional metric compatibility condition we will consider.  It is that the Levi-Civita connection for $g = \eta \otimes \eta'$ lies in the natural family of affine connections on the structure.  Specifically, this is the family of affine connections induced by connections on $S$ and $S'$,  each of which is flat on the respective top exterior power $\wedge^{2n}S$ or $\wedge^2S'$.  This condition is automatic when $n = 1$.  In \cite{BE91},  it is shown that for $n > 1$,  a metric that is so compatible will be Einstein.  \par  

We now introduce the notation and calculus that will enable much of our analysis.  The notation used is \textit{Penrose abstract index notation}\index{Penrose abstract index notation},  detailed in \cite{PR84v1,PR84v2}.  %Everything else in this section is detailed in \cite{BE91}.  \par 
Upper-case Latin indices should be interpreted as markers denoting the type of object,  and the repetition of indices denotes the natural pairing between a bundle and its dual.  The graphical ordering of indices is not important,  as the alphabetical ordering of indices keeps track of different factors in a tensor product.  \par 
Recall we will be considering complex manifolds $X$ with an identification
of $TX$ with a tensor product $S \otimes S'$.   Sections of the bundle $S$ and $S'$  are labelled by unprimed and primed indices respectively.  A vector field may therefore be written $V^{AA'} \in \Gamma(TX)$,  the metric tensor $g_{AA'BB'} \in \Gamma(\odot^2 T^*X)$,  and the torsion of an affine connection
\begin{align}
T_{ABA'B'}{}^{CC'} \in \Gamma(\wedge^2 T^*X \otimes TX),  
\end{align}
with $T_{ABA'B'}{}^{CC'} = -T_{BAB'A'}{}^{CC'}$.  \par 
Symmetrisation and skew-symmetrisation are represented by enclosing indices in round $()$ and square $[]$ brackets respectively.  Indices between delimiters $||$ inside brackets are skipped over when symmetrising.  For example,  given a connection $\nabla_{AA'}$ on $S$,  the \textit{twistor equation}\index{twistor!equation} on a section $o_{A} \in \Gamma(S)$ is
\begin{align}
\nabla_{(A|A'}o_{|B)} = 0.  
\end{align}
The symbol $\circ$ is used to denote the totally trace-free component of an object.  \par 
 On occasion when we prefer our indices to represent components of a section of a vector bundle with respect to some trivialisation,  we use lower-case Latin indices.  In this case,  unprimed and primed indices take values in sets $\{1,...,2n\}$ and $\{0',1'\}$ respectively.  \par 
The following is Theorem 2.4 of \cite{BE91} and the appropriate analogue of the fundamental theorem of Riemannian geometry. 
 \begin{prop}[Existence and uniqueness of adapted (Weyl) connections]\label{fundamental_theorem}
Given sections $\omega_{AB...C} \in \Gamma(\wedge^{2n}S^*)$ and $\eta_{A'B'} \in \Gamma(\wedge^2 S'^*)$ there is a unique pair of connections  on $S$ and $S'$ for which these two sections are parallel and such that the induced connection on $TX$ has torsion $T_{ABA'B'}{}^{CC'}$ that is totally trace-free.  Furthermore,  this totally trace-free torsion is an invariant of the structure.  
\end{prop}
The torsion always vanishes if $n=1$,  and then the pair of connections together induce the Levi-Civita connection for the holomorphic metric defined as the tensor product of the chosen sections.  
Given a choice of $\omega_{AB..C}$ and $\eta_{A'B'}$ as above,  we denote the corresponding connections by $\nabla_{AA'}$.  We call them Weyl connections\index{Weyl connections}\footnote{In the usual parlance of parabolic geometry,  these would be called \textit{exact} Weyl connections.}.  Later we will need the following formulae (see the more general Proposition \ref{general_change} for a proof)
\begin{prop}[Change of Weyl connections]\label{restricted_weyl}
The Weyl connections $\hat{\nabla}_{AA'}$ of Proposition \textup{\ref{fundamental_theorem}} corresponding to $\hat{\omega}_{AB...C} = \Omega \omega_{AB...C}$ and
$\hat{\eta}_{A'B'} = \Omega \eta_{A'B'}$ are:  
\begin{align}
\hat{\nabla}_{AA'}\mu^B \hspace{1mm} &= \nabla_{AA'}\mu^B \hspace{1mm}  + \delta_{A}^{B}\Upsilon_{CA'}\mu^{C}, \label{change_of_scales1}\\
\hat{\nabla}_{AA'}\mu^{B'} &= \nabla_{AA'}\mu^{B'} + \delta_{A'}^{B'}\Upsilon_{AC'}\mu^{C'}, \label{change_of_scales2}
\end{align}
where $\Upsilon_{AA'} = \nabla_{AA'} \log \Omega$.  
\end{prop}
From the above,  we may derive formulae for the change of the induced connections on the relevant associated bundles (tensor products and so on).  \par 
We will also utilise the following notation: We may write $\mathcal{O}[k]$ for the \textit{density bundle}\index{density bundle} $(\wedge^2 S')^k$.  For a vector bundle $E$ we write $E[k] := E \otimes (\wedge^2 S')^k$ and we say sections are ``of weight $k$".   There is a canonical section $\epsilon^{A'B'} \in \Gamma(\wedge^2S'[-1]) = \Gamma(\wedge^2S'^*[1])$ which we will use to raise and lower primed indices.  For example,  if $o^{A'} \in \Gamma(S'[k])$ then $o_{B'} := o^{A'}{\epsilon}_{A'B'} \in \Gamma(S'^*[k+1])$ and so on.  The reason for this contrivance is that the section is parallel for any Weyl connections,  so raising and lowering indices with this object commutes with $\nabla_{AA'}$,  and this leaves expressions unambiguous.  If $\omega \in \Gamma(\mathcal{O}[k])$ then (\ref{change_of_scales1}) implies
\begin{align}\label{change_of_scales3} 
\hat{\nabla}_{AA'}\omega = \nabla_{AA'}\omega + k\Upsilon_{AA'}\omega 
\end{align}
under a change of scale as in Proposition \ref{restricted_weyl}.  \par    
In addition to \cite{BE91},  a good primer on the calculus of ``weighted sections" is \cite{BEG94} (in the conformal case). 
\begin{defn}[Compatible metrics]\label{compat_def}
We say a holomorphic metric $g_{AA'BB'}$ is compatible if it is algebraically compatible,  that is 
\begin{align}
g_{AA'BB'} = \eta_{AB}\eta_{A'B'}
\end{align}
for some sections $\eta_{AB},  \eta_{A'B'}$ of  $\wedge^2 S^*$ and $\wedge^2 S'^*$,  and in addition there exist Weyl connections $\nabla_{AA'}$ inducing the Levi-Civita connection on $TX$.  
\end{defn}
Note that this implies compatible metrics only exist\index{compatible metrics} when the invariant torsion vanishes.  \par 
For $n \ge 2$,  these metrics are the complexified analogue of quaternion-K\"ahler \index{quaternion-K\"ahler metric}metrics\footnote{Often,  the definition of quaternion-K\"ahler metrics excludes Ricci-flat (hyper-K\"ahler) metrics.}: they are metrics with holonomy in $Sp(2n,\mathbb{C}) Sp(2,\mathbb{C}) \hookrightarrow SO(4n,\mathbb{C})$.  \par 
There is another important difference between the $n=1$ (holomorphic conformal geometry) and the $n>1$ settings.  While right-flatness for $n=1$ is equivalent to vanishing of the self-dual Weyl tensor,  when $n > 1$ right-flatness is equivalent to vanishing of part of the intrinsic torsion,  which implies the vanishing of part of the invariant curvature via a differential identity.  \par The following is from \cite{BE91}: 
\begin{prop}[Right-flatness]
When $n>1$,  a necessary and sufficient condition for every $\alpha$-plane to be tangent to an $\alpha$-surface \textup{(}right-flatness\textup{)} is
\begin{align}\label{right_flat_torsion}
T_{[AB](A'B')}{}^{CC'} = 0,
\end{align}
where $T_{ABA'B'}{}^{CC'}$ is the torsion of the induced connection on $TX$ corresponding to a \textup{(}equivalently,  any\textup{)} pair of Weyl connections $\nabla_{AA'}$ 
\end{prop}
So for $n > 1$,  right-flatness is required if there exists a compatible metric.  \par

Since we will be interested in Einstein metrics compatible with anti-self-dual conformal structures and more generally,  compatible metrics for almost-Grassmannian geometries,  we make the following definition: 
\begin{defn}[Torsion-free Grassmannian structure (TFG)]
We say that a $(2n,2)$-almost-Grassmannian structure is a torsion-free Grassmannian structure\index{torsion-free Grassmannian structure} \textup{(TFG)} if it is right-flat and torsion-free. 
\end{defn}
From now on we work in the setting of TFGs.  

\begin{prop}[Curvature of TFGs (\textup{\cite{BE91}} Eq.  13)]\index{torsion-free Grassmannian structure!curvature of}
Given a TFG for any $n \ge 1$,  the curvature of the connections $\nabla_{AA'}$ is determined by
\begin{align}
\nabla_{(A|A'}\nabla_{|B)}{}^{A'}\mu^{D} &= (\Psi_{ABC}{}^{D}- 2\Lambda_{C(A}\delta_{B)}{}^{D})\mu^{C}, \label{ASD_and_scalar_curvature} \\
\nabla_{(A|A'}\nabla_{|B)}{}^{A'}\nu^{B'} &= \Phi_{ABA'}{}^{B'}\nu^{A'}, \label{Einstein_curvature_primed} \\
\nabla_{[A}{}^{(A'}\nabla_{B]}{}^{B')}\mu^{D} &=  \delta_{[A}{}^{D}\Phi_{B]C}{}^{A'B'}\mu^{C},
 \label{Einstein_curvature_unprimed} \\
\nabla_{[A}{}^{(A'}\nabla_{B]}{}^{B')}\mu^{C'} &= \Lambda_{AB}\epsilon^{C'(A'}\mu^{B')}, \label{scalar_curvature}
\end{align}
where the objects on the right-hand side are sections of bundles
\begin{align}
\Phi_{(AB)(A'B')} &= \Phi_{ABA'B'} \in \Gamma(\odot^2 S^* \otimes \odot^2S'^*),  \\
\Psi_{(ABC)}{}^{D} &= \Psi_{ABC}{}^{D} \in \Gamma((\odot^3 S^* \otimes S)_{\circ}[-1]),\\
 \Lambda_{[AB]} &= \Lambda_{AB} \in \Gamma(\wedge^2 S^*[-1]).
\end{align}
$\Psi_{ABC}{}^{D}$ is invariant in the sense that it does not depend on the Weyl connections $\nabla_{AA'}$.  
\end{prop}
We define the Schouten tensor\index{Schouten tensor}:
\begin{align}
P_{AA'BB'} := \Phi_{ABA'B'} - \Lambda_{AB}\epsilon_{A'B'}.  
\end{align}
The reason for defining it is the neat rule for its transformation under a change in Weyl connections (\ref{change_of_scales1}) and (\ref{change_of_scales2}):
\begin{align}\label{schouten_change}
\hat{P}_{AA'BB'} = P_{AA'BB'} - \nabla_{AA'}\Upsilon_{BB'} + \Upsilon_{AB'}\Upsilon_{BA'}.  
\end{align}
All the above formulae will look familiar to practitioners of spinor calculus in four-dimensional anti-self-dual conformal geometry.  Indeed,  in the $n=1$ case they reduce to the formulae (2.6) and (2.7) in \cite{DT14}.  In particular,  the object $\Lambda_{AB}$ generalises the Ricci scalar.  $\Phi_{ABA'B'}$ is the analogue of the trace-free Ricci tensor.  Its vanishing is a generalisation of the Einstein condition.  Lastly,  the invariant piece $\Psi_{ABC}{}^{D}$ is the generalisation of the anti-self-dual Weyl tensor.  The general theory of parabolic geometries,  (see \cite{CS09}) implies the structure is locally isomorphic to $F_2(\mathbb{C}^{2n+2})$ given its vanishing.  \par
Note that the generalised Ricci-flatness condition ($\Phi_{ABA'B'} = 0,  \Lambda_{AB} = 0$) implies that the connection on $S'$ is flat.  The local holonomy is reduced to $SL(2n,\mathbb{C})$,  and it is equivalent to the admission of a parallel quaternionic structure.  If Weyl connections satisfying this condition are compatible with a metric,  then this is equivalent to the hyper-K\"ahler condition,  and the holonomy is further reduced to $Sp(2n,\mathbb{C})$. \index{hyper-K\"ahler}\index{Ricci-flat!scale}
\begin{exmp}[Hyper-Hermitian examples]\label{hyper_hermitian_example}
Define vector fields:
\begin{align}
U_{a} &= \frac{\partial}{\partial z^{a}} - \frac{\partial W^{b}}{\partial \theta^{a}}\frac{\partial}{\partial \theta^{b}},  \\ 
V_{a} &= \frac{\partial}{\partial \theta^{a}} .
\end{align}
We locally define vector bundles $S$ and $S'$ with trivialisations $\{\sigma_{a}\}_{a=1}^{2n}$ and $\{o',  \iota'\}$ respectively,  and define an isomorphism $TX \cong S \otimes S'$ by 
\begin{align}
U_{a} &\mapsto \sigma_{a} \otimes o',  \quad
V_{a} \mapsto \sigma_{a} \otimes \iota'.  \label{twistor_to_AG}
\end{align}
The right-flat condition is
\begin{align}
\frac{\partial W^{c}}{\partial z^{[a} \partial \theta^{b]}} - \frac{\partial W^{d}}{\partial \theta^{[a}}\frac{\partial^2 W^{c}}{\partial \theta^{b]} \partial \theta^{d}} = 0,  \label{hyper_hermitian_right_flat}
\end{align}
which is a system related to Pleba\'nski's second heavenly equation \textup{\cite{Pleb}}.  
We may write down the Weyl connections associated to the \textup{(}duals\textup{)} of $\sigma_1 \wedge ... \wedge \sigma_{2n}$,  and $o' \wedge \iota'$.  
Let $e_{a0'} = V_{a}$ and $e_{a1'} = U_{a}$.  
Assuming \textup{(\ref{hyper_hermitian_right_flat})},  one may check that the following connections on $S$ and $S'$,  with connection coefficients given in the trivialisations defined by $\{\sigma_{a}\}_{a=1}^{2n}$ and $\{o',\iota'\}$,  annihilate the scales and induce a torsion-free connection on $TX$.  
\begin{align}
\Gamma_{aa'b}{}^{c} &= \Bigg(\frac{\partial^2 W^{c}}{\partial \theta^{a} \partial \theta^{b}}-\frac{1}{2n+2}\frac{\partial ^2W^{d}}{\partial \theta^{a} \partial  \theta^{d}}\delta_{b}{}^{c}  - \frac{2}{2n+2}\frac{\partial ^2W^{d}}{\partial \theta^{b} \partial \theta^{d}}\delta_{a}{}^{c} \Bigg)\delta_{a'}{}^{1'}.  \\
\Gamma_{aa'b'}{}^{c'} &=-\frac{1}{2n+2} \Bigg(\epsilon_{a'b'}\frac{\partial^2 W^{d}}{\partial \theta^{a} \partial \theta^{d}}\delta_{0'}{}^{c'} + \frac{\partial^2 W^{c}}{\partial \theta^{a} \partial  \theta^{c}}\delta_{a'}{}^{c'}\delta_{b'}{}^{1'}\Bigg).  
\end{align}
Thus a solution $W^{a}$ to \textup{(\ref{hyper_hermitian_right_flat})} defines a TFG.  
We calculate the components $\Psi_{abc}{}^{d}$ of the invariant curvature as 
\begin{align}
\Psi_{abc}{}^{d} = \frac{\partial^3 W^d}{\partial \theta^{a} \partial \theta^{b} \partial \theta^{c}} - \frac{3}{2n+2}\frac{\partial^3 W^e}{\partial \theta^{e} \partial \theta^{(a} \partial \theta^{b}}\delta_{c)}{}^{d}.  
\end{align}
The construction and formulae here generalise the four-dimensional hyper-Hermitian conformal structures in \textup{\cite{D99}}.  \par 
One straightforward construction of a family of non-flat TFGs is the following: Take a non-zero constant,  degenerate symmetric tensor $T_{abc}$ and a non-zero vector $V^{a}$ such that $V^{a}T_{abc} = 0$.  Then taking 
\begin{align}\label{simple_soln}
W^{a} = V^{a}T_{bcd}\theta^{b}\theta^{c}\theta^{d}
\end{align}
defines a solution to \textup{(\ref{hyper_hermitian_right_flat})},  and the invariant curvature never vanishes.  The connection on $S'$ is flat in the given scale.  \par
\end{exmp}
Later,  we will discuss symmetries of Grassmannian structures.  Their generators generalise the notion of conformal Killing vector:
\begin{defn}[Infinitesimal symmetry]
An infinitesimal symmetry $W$ is a vector field $W \in \Gamma(TX)$ generating a one-parameter family of diffeomorphisms preserving the decomposition $TX \cong S \otimes S'$.  If $n = 1$,  we call an infinitesimal symmetry a conformal Killing vector.  
\end{defn}
Consider a one-parameter family $\phi_W^t$ of diffeomorphisms for $t \in (-\epsilon,\epsilon)$ preserving the cone of simple tensors.  Then we have,  for any decomposable element $\sigma(0) \otimes \tau(0) \in T_pX$.  
\begin{align}\label{cone_curve}
d\phi_W^t(\sigma(0) \otimes \tau(0)) = \sigma(t) \otimes \tau(t) 
\end{align}
for some curves $\sigma(t): (-\epsilon,\epsilon) \to S$,  $\tau(t): (-\epsilon,\epsilon) \to S'$,  lifting $\phi_W^t(p): (-\epsilon,\epsilon) \to X$.  Differentiating (\ref{cone_curve}) gives 
\begin{align}
[W,  \sigma \otimes \tau] = \Gamma \otimes \tau + \sigma \otimes \Pi.
\end{align}
In spinor notation the defining equation is 
\begin{align}
\nabla_{AA'}W^{BB'} = \phi_{A'}{}^{B'}\delta_{A}^{B} + \psi_{A}{}^{B}\delta_{A'}^{B'} + \frac{1}{2}\lambda\delta_{A}^{B}\delta_{A'}^{B'}
\end{align}
for trace-free $ \phi_{A'}{}^{B'},   \psi_{A}{}^{B}$.  \par  
In the case there is a compatible metric in the scale,  then the above with $\lambda = 0$,  and $\psi_{[A}{}^C\eta_{B]C} = 0$ (the second condition is automatic when $n=1$) defines the \textit{Killing equation} for that metric.  
\section{Conformally Einstein anti-self-dual metrics}\label{KNT_repeat}
In this section we consider the $n=1$ case,  when a TFG is the same as a holomorphic conformal structure \index{conformal structure} \index{ASD conformal structure} $(X,[g])$ with anti-self-dual Weyl tensor.  In particular,  we ask when such a conformal structure has a scale in which the metric is Einstein.  The arguments in this section are a straightforward modification of those found in \cite{KNT85} to the anti-self-dual setting,  but provide a good exercise before we proceed to generalise to higher dimensions.  \par  
Because $n = 1$,  a choice of scales $\eta_{AB} \in \Gamma(\wedge^2S^*),  \eta_{A'B'}  \in \Gamma(\wedge^2S'^*)$ is the same as a choice of metric
\begin{align}
g_{AA'BB'} = \eta_{AB}\eta_{A'B'} 
\end{align}
in the conformal class.  The induced connection on $TX$ is obviously metric and we may decompose the torsion,  using the fact $n=1$,  as
\begin{align}\label{torsion_asd}
T_{ABA'B'}{}^{CC'} = \epsilon_{AB}F_{A'B'}{}^{CC'} + \epsilon_{A'B'}G_{(AB)}{}^{CC'} 
\end{align}
for $F_{A'B'}{}^{CC'} \in \Gamma(\odot^2 S'^* \otimes TX),  G_{AB}{}^{CC'} \in \Gamma(\odot^2 S^* \otimes TX)$.  That the summands on the right-hand side are trace-free implies they vanish. 
So the Weyl connections always induce the Levi-Civita connection.  \par To obtain all the metrics in the conformal class we need only consider simultaneously rescaling $\eta_{AB}$ and $\eta_{A'B'}$ by $\Omega$ and hence the formulae (\ref{change_of_scales1}) and (\ref{change_of_scales2}).  In other words we may,  without loss of generality,  choose an isomorphism $\wedge ^2 S^* \cong \wedge ^2 S'^*$.  We can then raise and lower indices with canonical sections $\epsilon^{AB} \in \Gamma(\wedge^2S[-1])$ and $\epsilon^{A'B'} \in \Gamma(\wedge^2 S'[-1])$.  \par
Note that a choice of $\sigma \in \Gamma(\mathcal{O}[1])$ is the same as a choice of conformal metric: 
\begin{align}
g_{AA'BB'} = \sigma^{-2}\epsilon_{AB}\epsilon_{A'B'}.  
\end{align}
The conformal-to-Einstein condition can then be written as the existence of $\sigma \in \Gamma(\mathcal{O}[1])$ satisfying 
\begin{align}\label{cte}
(\nabla_{AA'}\nabla_{BB'}
+ P_{AA'BB'})_\circ\sigma = 0,  
\end{align}
as in \cite{BEG94}.  \par 
All the formulae in this section should be consistent with the conventions used in \cite{PR84v1}, \cite{PR84v2},  \cite{DT14}. 
\begin{defn}[Weyl invariant densities]
\index{Weyl invariants}
Given a holomorphic anti-self-dual conformal structure $(X,[g])$ define $I \in \Gamma(\mathcal{O}[-2])$ and $J \in \Gamma(\mathcal{O}[-3])$ by
\begin{align}
\Psi_{ABCD}\Psi^{ABCE} &= \frac{1}{2}I\delta_{D}{}^{E},  \\
\Psi_{ABCD}\Psi^{AB}{}_{PQ}\Psi^{PQCE} &= \frac{1}{2}J\delta_{D}{}^{E}.  
\end{align}
\end{defn}
Note that in fact $I = \Psi_{ABCD}\Psi^{ABCD}$ and  $J = \Psi_{ABCD}\Psi^{AB}{}_{EF}\Psi^{EFCD}$.  
The non-vanishing of $I$ implies the non-degeneracy of the map $S^* \to \odot^3 S^*[-1]$ defined by $\Psi_{ABC}{}^{D}$.  
The non-vanishing of $J$ is equivalent to the non-degeneracy of the $\mathcal{O}[-2]$-valued endomorphism of $\odot^2 S^*$ defined by $\Psi_{AB}{}^{CD}$.  \par
From the Bianchi identities,  the divergence of the Weyl tensor on an Einstein manifold vanishes.  So if $\hat{\nabla}_{AA'}$ corresponds to an Einstein scale,  then
\begin{align}\label{cspace}
\hat{\nabla}_{AA'}\Psi_{BCD}{}^{A} = 0.  
\end{align}
Recalling the ASD Weyl tensor has weight $-1$,  we use (\ref{change_of_scales1},  \ref{change_of_scales3}),  and conclude the following necessary condition for the existence of an Einstein scale: In an arbitrary scale there exists a closed $1$-form $\Upsilon_{AA'}$ such that 
\begin{align}\label{C_space}
\nabla_{AA'}\Psi_{BCD}{}^{A} + \Upsilon_{AA'}\Psi_{BCD}{}^{A} = 0.
\end{align}
If $\Upsilon_{AA'} = d \log \Omega$,  then in the scale $\hat{\eta}_{A'B'} = \Omega \eta_{A'B'}$,  (\ref{cspace}) will hold.  \par 
\begin{prop}
Let $(X,[g])$ be an anti-self-dual conformal structure such that $J \ne 0$.  If there exists a $1$-form $U_{AA'} \in \Gamma(T^*X)$ such that 
\begin{align}\label{weak_C_space}
\nabla_{AA'}\Psi_{BCD}{}^{A} + U_{AA'}\Psi_{BCD}{}^{A} = 0
\end{align}
then $U_{AA'}$ is closed.   
\end{prop}
\begin{proof}
Differentiate (\ref{weak_C_space}) by $\nabla^{BA'}$ and use the equation again to obtain
\begin{align}
\nabla_{B}{}^{A'}\nabla_{AA'}\Psi_{CD}{}^{AB} = U_{AA'}U_{B}{}^{A'}\Psi_{CD}{}^{AB} -(\nabla_{B}{}^{A'}U_{AA'})\Psi_{CD}{}^{AB}.  
\end{align}
The left-hand side can be rewritten in terms of the curvature formula (\ref{ASD_and_scalar_curvature}).  It vanishes,  which one can see either by direct calculation or observing there are no non-vanishing objects in $\odot^2 S^*[k]$ that are quadratic in $\Psi_{ABCD}$ or formed from a contraction of $\Psi_{ABCD}$ and $\Lambda_{AB}$.  
The first term on the right-hand side vanishes by symmetry considerations.  Therefore
\begin{align}
(\nabla_{B}{}^{A'}U_{AA'})\Psi_{CD}{}^{AB} = 0.  
\end{align}
Assuming $J \ne 0$ this implies $\nabla_{(B}{}^{A'}U_{A)A'} = 0$.  
We now need to handle the self-dual derivative $\phi_{A'B'} := \nabla_{A(A'}U^{A}{}_{B')}$.  The key fact is that an anti-self-dual conformal structure has \textit{Bach tensor}\index{Bach tensor} that vanishes identically \cite{BE90,  BM87}.  That is,  we may utilise an identity:
 \begin{align}
 \nabla_{C}{}^{A'}\nabla_{E}{}^{C'}\Psi_{AB}{}^{CE} + \Phi_{CE}{}^{A'C'}\Psi_{AB}{}^{CE} = 0. 
 \end{align}
 Now differentiate the condition (\ref{weak_C_space}) and use the above Bach-vanishing condition to obtain
 \begin{align}
(U_{(C}{}^{A'}U_{D)}{}^{C'} - \nabla_{(C}{}^{(A'}U_{D)}{}^{C')} + \Phi_{CD}{}^{A'C'})\Psi_{AB}{}^{CD} = 0.  
 \end{align}
When $J \ne 0$ this implies
 \begin{align}\label{EW}
U_{(A}{}^{A'}U_{B)}{}^{B'} - \nabla_{(A}{}^{A'}U_{B)}{}^{B'} + \Phi_{AB}{}^{A'B'} = 0.
 \end{align}
 This is the statement that the connection $\mathcal{D}_{AA'} := \nabla_{AA'} + U_{AA'}$ is Einstein-Weyl.  We now use the result of Calderbank \cite{Calderbank} that on an anti-self-dual manifold,  the exterior derivative of $U_{AA'}$ (the so-called Faraday curvature) is anti-self-dual.  The identity can be checked easily by differentiating the Bianchi identity
 \begin{align}
 \mathcal{D}_{A(A'}\phi_{B'C')} = 0.  
 \end{align}
 Thus $\phi_{A'B'} = 0$ and $U_{AA'}$ is a closed 1-form.  
 \end{proof}
 See also Theorem 4.2 of \cite{NG},  which implies that in Riemannian signature (with no assumptions on $I,J$) an algebraic solution to (\ref{weak_C_space}) in the ASD setting is closed.   \par 
The automatic vanishing of the Bach tensor means that in our ASD setting the existence of a $\Upsilon_{AA'}$ solving (\ref{C_space}) is then sufficient for the existence of an Einstein scale (when $J \ne 0$):
 \begin{prop} \label{C-space equivalent to conformal Einstein generically} 
Let $(X,[g])$ be an anti-self-dual conformal structure with $J \ne 0$.  If there exists a \textup{(}necessarily closed\textup{)} $1$-form $\Upsilon_{AA'} \in \Gamma(T^*X)$ such that \textup{(\ref{C_space})} holds,  then there exists an Einstein metric in the conformal class.  
 \end{prop}
\begin{proof}
We return to (\ref{EW}) with $U_{AA'} = \Upsilon_{AA'}$ exact.  
Letting $\Omega$ be the function such that $\Upsilon_{AA'} = d \log \Omega$ we have,  for the Weyl connections corresponding to $\hat{\eta}_{AB} = \Omega \eta_{AB}$,  $\hat{\eta}_{A'B'} = \Omega \eta_{A'B'}$
\begin{align}
\hat{\Phi}_{ABA'B'} = 0,
\end{align}
which is the Einstein condition.  
\end{proof}
%In the real,  Lorentzian setting,  vanishing of the Bach tensor is an additional condition whereas it is automatic in the anti-self-dual case.  \par 
If $I \ne 0$,  if there exists a not necessarily closed $U_{AA'}$ solving (\ref{weak_C_space}),  it is unique,  and furthermore,  there is a formula (irrespective of whether $J$ vanishes): Contracting both sides of (\ref{weak_C_space}) with $\Psi^{BCD}{}_{E}$ gives 
\begin{align}\label{BE_solve}
\Psi^{BCD}{}_{A}\nabla_{EA'}\Psi_{BCD}{}^{E} = \frac{I}{2}U_{AA'},  
\end{align}
and this yields a formula for $U_{AA'}$.  \par 
Substituting this into (\ref{weak_C_space}) therefore yields necessary and sufficient conditions that (\ref{weak_C_space}) has an algebraic solution.  This was the necessary condition for an Einstein scale identified in \cite{BE90} which we have here shown is sufficient:
\begin{thm}\label{generic_condition}
Suppose that $(X,[g])$ is a \textup{(}holomorphic\textup{)} anti-self-dual conformal structure with $I \ne 0$ and $J \ne 0$.  The conformal class contains an Einstein metric if and only if 
\begin{align}\frac{I}{2}\nabla^{E}{}_{A'}\Psi_{ABCE} - \Psi^{FGH}{}_{D}(\nabla_{EA'}\Psi_{FGH}{}^{E})\Psi_{ABC}{}^{D} &= 0,   \label{BEinvariant} 
\end{align}
\end{thm}
The obstruction (\ref{BEinvariant}) is a section of the bundle $S'^* \otimes \odot^3 S^*$ so defines,  after choosing a basis,  eight scalar equations.  However,  only four are independent since by construction  (\ref{BEinvariant}) lies in the kernel of the map $S'^* \otimes \odot^3 S^* \to S'^* \otimes S$ defined by $\Psi^{ABC}{}_{D}$.  \par 
For completeness,  if we remove the hypothesis $J \ne 0$,  then necessary and sufficient conditions are obtained,  for example,  in \cite{GN06}.  They are equivalent to the vanishing of a section of $\End_\circ(TX)$:
\begin{thm}\label{degen_condition}
Suppose that $(X,[g])$ is a \textup{(}holomorphic\textup{)} anti-self-dual conformal structure with $I \ne 0$.  The conformal class contains an Einstein metric if and only if
\begin{align}\label{Einstein rescale}
 U_{(A}{}^{A'}U_{B)}{}^{B'} - \nabla_{(A}{}^{(A'}U_{B)}{}^{B')} + \Phi_{AB}{}^{A'B'} &= 0 \\
 dU &= 0. 
\end{align}
where $U_{AA'}$ is defined by \textup{(\ref{BE_solve})}.  
\end{thm}
Note that the conditions in Theorem \ref{degen_condition} require calculating the fourth derivatives of the metric components,  one order higher than required for those in Theorem \ref{generic_condition}.  In the non-ASD case,  one needs to check the vanishing of the Bach tensor,  which is also fourth order in the metric. 
\begin{exmp}[Araneda's ASD Pleba\'nski-Demia\'nski metrics]
The Pleba\'nski-Demia\'nski \textup{\cite{PD76}} family of metrics is given by,  in coordinates $(\tau,\phi,p,q)$
\begin{align}\label{bernardo}
g = \frac{1}{(p-q)^2}\Big[ -Q \frac{(d\tau - p^2 d\phi)^2}{(1-p^2q^2)} + P\frac{(d\phi - q^2 d\tau)^2}{(1-p^2q^2)} + (1-p^2q^2)\bigg( \frac{dp^2}{P} - \frac{dq^2}{Q}\bigg) \Big],  
\end{align}
for arbitrary $P(p),  Q(q)$.  
Araneda \textup{\cite{A25}} showed that the metric is anti-self-dual\,\footnote{Here we adopt the opposite orientation from \cite{A25}} if and only if
\begin{align}
P = a_0 + a_1p + a_2p^2 + a_3p^3 + a_4p^4 \\
Q = a_4 + a_3q + a_2q^2 + a_1q^3 + a_0q^4
\end{align}
for constants $a_0,a_1,a_2,a_3,a_4 \in \mathbb{C}$.  This family of metrics provides a prototype \textup{\cite{AAS26}} for recent advances in twistorial descriptions of some black hole spacetimes \textup{\cite{AASS26}}.  
Computing with the Maple \texttt{DifferentialGeometry} package we obtain:
\begin{align}
I = \frac{6(a_1q - a_3p)^2(p-q)^4}{(qp+1)^6}
\end{align}
and $J$ is determined by the fact the metric is Petrov type D.  
The obstruction \textup{(\ref{BEinvariant})} vanishes and the relevant $1$-form is
\begin{align}
U = \frac{2q(a_1-a_3)dp}{(a_1q-a_3p)(p-q)} - \frac{2p(a_1-a_3)dq}{(a_1q-a_3p)(p-q)}.
\end{align}
Hence, 
\begin{align}
\tilde{g} = \frac{(p-q)^2}{(a_1q - a_3p)^2}g
\end{align}
is Einstein.  \par 
We could have concluded this by other means: Proposition \textup{7} of \textup{\cite{ACG}} implies that any four-dimensional ambi-K\"ahler metric \textup{(}a metric that is conformal to a pair of K\"ahler metrics,  oppositely oriented from one another\textup{)} with vanishing Bach tensor is conformal to an Einstein metric.  The ambi-K\"ahler structure is shown in \textup{\cite{A25}}.  
\end{exmp}
Not all algebraic solutions of (\ref{weak_C_space}) are closed.  The following example demonstrates the Riemannian assumption in Theorem 4.2 of \cite{NG} is essential:
\begin{exmp}[Nil Einstein-Weyl monopole]
There is a well-known correspondence between three-dimensional Einstein-Weyl spaces $(h,\omega)$ equipped with solutions to the monopole equation and ASD conformal structures with a non-null conformal Killing vector,  established in \textup{\cite{JT}}.  If the Einstein-Weyl space is taken to be hyperbolic space $\mathbb{H}^3$,  then Proposition \textup{3.1} in \textup{\cite{PT}} is that the resulting ASD conformal structure is conformally Einstein if and only if the solution to the monopole equation is spherically symmetric. 
Another example of an Einstein-Weyl space exists on the Lorentzian Nil space.  
The Einstein-Weyl metric and $1$-form are given by
\begin{align}
h = (dy + xdt)^2 - 4dxdt,  \quad \omega = dy + xdt.  
\end{align}
It belongs to a class that arises from the reduction of the Einstein-Weyl equations to the $\text{Diff}(S^1)$ equation \textup{\cite{D04}}.  
We consider the monopole $(F,\beta)$ where $F$ is a function and $\beta$ is a $1$-form satisfying
\begin{align}
* (dF + \frac{1}{2}\omega F) = d\beta.  
\end{align}
A straightforward family of solutions is 
\begin{align}
F(t,x,y) &= V(x) + W(t) \\
 \beta &= -(yW'(t) + xV(x) + xW(t))dt + yV'(x)dx.  
\end{align}
The corresponding ASD metric with symmetry $\partial_\tau$ is given by
\begin{align}
g = F^2\big((dy + xdt)^2 - 4dxdt\big) - (d\tau  + \beta )^2.  
\end{align}
If we set $F = W(t) + c$ then $I = J = 0$ but the metric is not Petrov type N.  There is a unique algebraic solution to \textup{(\ref{weak_C_space})} which is:
\begin{align}
U = \Bigg(
\frac{W''}{3W'}
+\frac{(y-6)W'}{4(c+W)}
\Bigg) dt
-\frac14\, dy
-\frac{1}{4(c+W)}\, d\tau 
\end{align}
This is not closed unless $W$ is a constant,  in which case $g$ is conformally flat.  
\end{exmp}
Necessary and sufficient conditions for the existence of a Ricci-flat conformal scale \index{Ricci-flat!metric} are:
\begin{thm}\label{ricci_flat_condition}
Suppose that $(X,[g])$ is a \textup{(}holomorphic\textup{)} anti-self-dual conformal structure with $I \ne 0$.  The conformal class contains a Ricci-flat metric if and only if it contains an Einstein metric and 
\begin{align}\label{Scalar_kill}
4\Lambda + \nabla_{AA'}U^{AA'} + U_{AA'}U^{AA'} = 0,  
\end{align}
where $U_{AA'}$ is defined by \textup{(\ref{BE_solve})} and $2\Lambda := \Lambda_{AB}\epsilon^{AB}$.  
\end{thm}
These are fourth order in the metric components.  \par 
In \cite{DT14}, necessary and sufficient conditions for the existence of a Ricci-flat metric in a Riemannian conformal class are given.  
If $J \ne 0$,  Theorem \ref{generic_condition} in this article implies that all but the trace of the tensor invariant (1.4) of Theorem 1.1 in  \cite{DT14}  vanish automatically given vanishing of the scalar invariant (1.3).  That scalar invariant vanishes with the square of the norm of the invariant (\ref{BEinvariant}) here.  
 \par 
\begin{rem}
We briefly mention some other perspectives on ASD-Einstein metrics: There is a correspondence \textup{\cite{W80}} between holomorphic contact structures on the twistor space,  transverse to twistor lines and holomorphic Einstein metrics with non-zero cosmological constant \textup{(}up to scale\textup{)}.  
The existence of such an Einstein scale may also be equivalently stated in terms of the existence of an adapted coordinate system and scalar potential satisfying a second-order non-linear PDE \textup{\cite{PrzanKill,CP13,DM26}}.  The antecedents in the Ricci-flat case are the original non-linear graviton construction \textup{\cite{P76}} and the formalism of Pleba\'nski \textup{\cite{Pleb}}.  
\end{rem}
\section{ASD-K\"ahler four-manifolds}\label{ASDK}
The ASD-K\"ahler metrics we will consider are ones for which the K\"ahler form defines a self-dual $2$-form.  Otherwise,  the behaviour of these structures is different: the ASD Weyl curvature is restricted in the oppositely oriented case whereas it can be algebraically general for the metrics we consider here.  \par 
A conformally K\"ahler structure corresponds to a section $\omega_{A'B'} \in \Gamma(\odot^2S'^*[2])$ (which determines a self-dual $2$-form in a scale) satisfying $\omega_{A'B'}\omega^{A'B'} \ne 0$ and the conformally invariant overdetermined PDE 
\begin{align}\label{Kahler PDE}
\nabla_{A(A'}\omega_{B'C')} = 0.    
\end{align}
This is precisely the integrability condition for the existence of a scale $\hat{\sigma} \in \Gamma(\mathcal{O}[1])$ such that $\hat{\nabla}_{AA'}\omega_{B'C'} = 0$,  and in this scale $\hat{\sigma}^{-1}\delta_{A}{}^{B}\omega_{A'}{}^{B'} \in \Gamma(\operatorname{End}(TX))$ may be normalised to define a compatible complex structure.  \par 
The curvature formula (\ref{scalar_curvature}) shows that an ASD-K\"ahler metric is scalar-flat,  and in fact all scalar-flat K\"ahler metrics are ASD-K\"ahler.  Another integrability condition in the K\"ahler scale is that
\begin{align}\label{Ricci_spinor}
\Phi_{ABA'B'} = \phi_{AB}\omega_{A'B'}
\end{align}
for $\phi_{AB} \in \Gamma(\odot^2 S^*[-2])$,  which we call the Ricci spinor.  \par 
The question of which ASD-Einstein metrics with non-zero cosmological constant were conformal to ASD-K\"ahler metrics was answered in \cite{T06,  DT10}.  Here,  we solve the reverse problem and give necessary and sufficient conditions for an ASD-K\"ahler structure to be conformally Einstein.  \par 
First,  we prove the result of \cite{DT10} (extended to the indefinite setting):
\begin{thm}[\cite{DT10}]
An ASD-Einstein metric with non-zero cosmological constant is conformal to an  ASD-K\"ahler metric on an open dense set if and only if it has a Killing vector that is not everywhere null.  
\end{thm}
\proof 
The PDE (\ref{Kahler PDE}) is equivalent to the existence of $U_{AA'} \in \Gamma(T^*X[1])$ satisfying 
\begin{align}\label{ctk}
\nabla_{AA'}\omega_{B'C'} = \epsilon_{A'B'}U_{C'A} + \epsilon_{A'C'}U_{B'A}.  
\end{align}
We may rule out $U_{AA'} = 0$,  since if it were,  the metric would already be K\"ahler and so Ricci-flat.   \par 
Using the Einstein condition,  differentiating and commuting derivatives yields:
\begin{align}\label{Killing_equation_Kahler}
\nabla_{AA'}U_{BB'} = -\Lambda\epsilon_{AB}\omega_{A'B'} + \epsilon_{A'B'}\rho_{AB} 
\end{align}
for some $\rho_{AB} \in \Gamma(\odot^2 S^*).$  Letting $\sigma \in \Gamma(\mathcal{O}[1])$ be the parallel section in the Einstein scale,  the above implies $K_{AA'} := \sigma^{-1} U_{AA'}$ defines (the dual of) a Killing vector.   Suppose $K^{AA'}$ is everywhere null.  Write $\Box := \nabla_{AA'}\nabla^{AA'}$.  We have
\begin{align}\nonumber 
0 = \frac{1}{2}\Box (U_{BB'}U^{BB'})
&=
(\nabla_{AA'}U_{BB'})
(\nabla^{AA'}U^{BB'})
+
U^{BB'}\Box U_{BB'}
\\
&=
(\nabla_{AA'}U_{BB'})
(\nabla^{AA'}U^{BB'}),  
\end{align}
where we have used $\Box U_{AA'}-6\Lambda\epsilon_{AB}\epsilon_{A'B'}U^{BB'} = 0$,  which may be derived by taking curvature on $U_{AA'}$.  This implies $\Lambda^2\omega_{A'B'}\omega^{A'B'} + \rho_{AB}\rho^{AB} = 0$.  On the other hand nullity implies the left-hand side of (\ref{Killing_equation_Kahler}) vanishes upon contraction with $U^{BB'}$. This implies (\ref{Killing_equation_Kahler}) defines a degenerate $2$-form and hence $\Lambda^2\omega_{A'B'}\omega^{A'B'} - \rho_{AB}\rho^{AB} = 0$.  Together we have $\omega_{A'B'}\omega^{A'B'} = 0$ which contradicts the desired K\"ahler condition.  So $K^{AA'}$ is not everywhere null.  \par
\par Conversely,  let us consider the Killing equation in an Einstein scale.  If there exists $U_{AA'} \in \Gamma(T^*X[1])$ defining a Killing vector then we first claim its self-dual derivative is non-zero.  If 
$\nabla_{AA'}U_{BB'} = \epsilon_{A'B'}\rho_{AB} $,  
then differentiating and taking curvature implies $\Lambda = 0$,  contradicting our assumption.   Therefore,  we may take $U_{AA'}$ to satisfy (\ref{Killing_equation_Kahler}) for some $\rho_{AB}  \in \Gamma(\odot^2 S^*)$ and non-zero $\omega_{A'B'} \in \Gamma(\odot^2S'^*[2])$.  Next,  
\begin{align}
\nabla_{B(A'}\nabla^{B}{}_{B'|}U_{A|C')} = \Lambda\nabla_{A(A'}\omega_{B'C')},  
\end{align}
but the left hand side is zero by the formulae (\ref{Einstein_curvature_unprimed}) and (\ref{scalar_curvature}).  Therefore $\nabla_{A(A'}\omega_{B'C')} = 0$,  since $\Lambda \ne 0$.  Lastly,  suppose that $\omega_{A'B'}$ was degenerate,  then $\omega_{A'B'} = \psi_{A'}\psi_{B'}$ for some $\psi_{A'}$.  Contracting (\ref{ctk}) with $\psi^{A'}\psi^{B'}$ yields $\psi^{B'}U_{B'A} = 0$,  which would imply the Killing vector is everywhere null.  
 \qed   \par
The Killing vector field actually remains a bona fide Killing vector in the new scale.  To see this note that the K\"ahler scale is defined by the $1$-form $\Upsilon_{AA'}$ where
\begin{align}\label{hamilton} U_{AA'}=-\omega_{A'}{}^{B'}\Upsilon_{AB'},\end{align}  
so 
\begin{align}
U^{AA'}\Upsilon_{AA'}
&=
-\omega^{A'}{}_{B'}\,
  \Upsilon^{AB'}\Upsilon_{AA'} = 0.  
\end{align}
  So the conformal factor is constant along the Killing vector and it remains Killing in the K\"ahler scale.  The equation (\ref{hamilton}) implies the Killing vector is a Hamiltonian vector field.  
  
Furthermore,  define $L_{AA'} = \Omega U_{AA'} \in \Gamma(T^*X[1])$.  Let $\hat{g}  = \Omega^2 g$ be the K\"ahler metric.  Now $\hat{\sigma} = \Omega^{-1} \sigma$ is the parallel section of $\mathcal{O}[1]$ in the K\"ahler scale.  Then $\hat{\sigma}^{-1}L_{AA'}$ is dual to the $1$-form of the Killing vector with respect to the K\"ahler metric.  The formulae (\ref{change_of_scales1},  \ref{change_of_scales2}) give:
\begin{align}\label{new_killing}
\hat{\nabla}_{AA'}L_{BB'} &= \Omega\nabla_{AA'}U_{BB'} + \Upsilon_{AA'}L_{BB} - \Upsilon_{BB'}L_{AA'}  + \Upsilon^{CC'}L_{CC'}\epsilon_{AB}\epsilon_{A'B'} \\ \nonumber
&= -\Omega\Lambda\epsilon_{AB}\omega_{A'B'} + \Omega\epsilon_{A'B'}\rho_{AB} + \Upsilon_{AA'}L_{BB'} - \Upsilon_{BB'}L_{AA'}.  
\end{align}
Now use (\ref{hamilton}) to get
\begin{align}
\hat\nabla_{AA'}L_{BB'}
=
{\Omega
(
-\Lambda+\frac12\Upsilon_{CC'}\Upsilon^{CC'})
\epsilon_{AB}\omega_{A'B'} }  + \Omega
\epsilon_{A'B'}
(
\rho_{AB}
-\Upsilon_{(A}{}^{C'}U_{B)C'}
). 
\end{align}
Extracting the anti-self-dual part of (\ref{Killing_equation_Kahler}) we may rewrite: 
\begin{align}
\rho_{AB}
-\Upsilon_{(A}{}^{C'}U_{B)C'}
&=
\frac12\nabla_{(A|C'|}U_{B)}{}^{C'}
-\Upsilon_{(A}{}^{C'}U_{B)C'} \\ \nonumber 
&= 
-\frac12\omega^{C'D'}
\left(
\nabla_{(A|C'|}\Upsilon_{B)D'}
-\Upsilon_{AC'}\Upsilon_{BD'}
\right).
\end{align} 
In particular,  the expression in brackets is precisely the Ricci curvature in the K\"ahler scale,  and from (\ref{Ricci_spinor}) we see in fact that the anti-self-dual part of $\hat{\nabla}_{AA'}L_{BB'}$ is pointwise proportional to $\epsilon_{A'B'}\phi_{AB}$.  \par
The conditions we have found above on the Killing vector are sufficient for it to determine an Einstein scale:
\begin{thm}\label{scalar_flat_kahler_Einstein_criterion}
An ASD-K\"ahler metric with non-degenerate Ricci spinor:
\begin{align}
\phi_{AB}\phi^{AB} \ne 0
\end{align}
is conformal to an Einstein metric on an open dense set if and only if it has a Killing vector field preserving the K\"ahler structure and
anti-self-dual derivative proportional to $\epsilon_{A'B'}\phi_{AB}$.   Explicitly,  the condition is that
\begin{align}\label{scalar_flat_kahler_conformally_Einstein_pde}
\nabla_{AA'}K_{BB'}
=
\lambda\epsilon_{AB}\omega_{A'B'}
+
\sigma\epsilon_{A'B'}\phi_{AB}
\end{align}
for some $\lambda \in \Gamma(\mathcal{O}[-3]),  \sigma \in \Gamma(\mathcal{O}[1])$,  where $\nabla_{AA'}$ is the Levi-Civita connection and $K_{AA'} \in \Gamma(T^*X)$ is the dual of the Killing vector,  both taken with respect to the K\"ahler scale.  
\end{thm}
\begin{proof}
The K\"ahler condition implies $\omega_{A'B'}$ is normalised according to
\begin{align}
\omega_{A'B'}\omega^{A'B'}=2,
\qquad
\omega_{A'}{}^{B'}\omega_{B'}{}^{C'}
=-\delta_{A'}{}^{C'}.
\end{align}
The forward implication was proved above in the case the Einstein scale was non-Ricci-flat but can in general be readily seen,  using the fact the K\"ahler form is parallel and contracting it with the conformal-to-Einstein equation (\ref{cte}).  \par
 Conversely,  suppose that
\eqref{scalar_flat_kahler_conformally_Einstein_pde} holds.
Differentiating
(\ref{scalar_flat_kahler_conformally_Einstein_pde}),
taking curvature, and equating the SD and ASD parts to zero gives
\begin{align}
\omega_{C'}{}^{A'}\nabla_{BA'}\lambda
-3\phi_B{}^A\nabla_{AC'}\sigma
&=-2\phi_B{}^A\omega_{C'}{}^{A'}K_{AA'},\\
\phi_B{}^A\nabla_{AC'}\sigma
-3\omega_{C'}{}^{A'}\nabla_{BA'}\lambda
&=-2\phi_B{}^A\omega_{C'}{}^{A'}K_{AA'},
\end{align}
using that $\omega_{A'B'}$ is parallel,  and the contracted Bianchi identity $
\nabla_{AC'}\phi_B{}^A=0$.  These yield:
\begin{align}
\omega_{C'}{}^{A'}\nabla_{BA'}\lambda
=
\phi_B{}^A\nabla_{AC'}\sigma.
\end{align}
Back-substitution into either equation then gives
\begin{align}
\phi_B{}^A
(
\nabla_{AC'}\sigma
-\omega_{C'}{}^{A'}K_{AA'}) = 0.  
\end{align}
The non-degeneracy of $\phi_{AB}$ implies
\begin{align}
\nabla_{AA'}\sigma
=
\omega_{A'}{}^{B'}K_{AB'}.
\end{align}
Note that this implies $\sigma$ is not identically zero,  otherwise $K_{AA'} = 0$.  \par 
Back-substitute into (\ref{scalar_flat_kahler_conformally_Einstein_pde}):
\begin{align}
\nabla_{AA'}\nabla_{BB'}\sigma
+\sigma\phi_{AB}\omega_{A'B'}
=
-\lambda\epsilon_{AB}\epsilon_{A'B'}.  
\end{align}
The trace-free part is the condition that $\sigma$ defines an Einstein scale where $\sigma \ne 0$.  
\end{proof} 
Evidently,  the above proof applies without modification in the real setting.  In Riemannian signature,  $\phi_{AB}\phi^{AB} \ne 0$ is just the statement the metric is not Ricci-flat.  The K\"ahler-preserving condition is that the Killing vector field is holomorphic for the complex structure.  
\begin{thm}
A Riemannian ASD-K\"ahler metric,  not Ricci-flat,  is conformal to an Einstein metric on an open dense set if and only if it has a holomorphic Killing vector field with
anti-self-dual derivative proportional to the Ricci spinor.  
\end{thm}
\begin{rem}
An argument using the Hitchin-Thorpe inequality implies there are no compact examples of ASD-K\"ahler  four-manifolds,  not Ricci-flat,  which are conformally Einstein by a global conformal factor \textup{\cite{LB}}.  
\end{rem} 
\begin{exmp}[Araneda's ASD Pleba\'nski-Demia\'nski metrics,  revisited]
It is shown in \textup{\cite{A25}} that the metrics \textup{(\ref{bernardo})} are conformally K\"ahler.  Take:
\begin{align}
g_{-}  = \frac{(p-q)^2}{(1 - pq)^2}g. 
\end{align}
It is K\"ahler with the K\"ahler form
\begin{align}
\omega_{-} = \frac{(d\tau - p^2 d\phi) \wedge dq + dp \wedge (d\phi - q^2 d\tau)}{(1-pq)^2}. 
\end{align}
A Killing vector field $K$ of Theorem \textup{\ref{scalar_flat_kahler_Einstein_criterion}} is obtained by solving  $\iota_K \omega_{-} = d\Omega^{-1}$ where the Einstein representative is given by $\tilde{g} = \Omega^2 g_{-}$.  It is 
\begin{align}
K = a_{1}\frac{\partial}{\partial \tau} + a_3 \frac{\partial}{\partial \phi}.  
\end{align}
\end{exmp}
\section[The metrisability equation for torsion-free Grassmannian structures]{A metrisability equation for torsion-free Grassmannian structures}\label{metrisability}
In this section we consider the conditions for an almost-Grassmannian structure to admit a compatible metric as per Definition \ref{compat_def}.  
As we explained earlier,  we must restrict ourselves to TFGs,  and the condition is vacuous for $n = 1$.  \par 
It was first shown (for a real analogue of the structures we are considering here) in \cite{F16} that in the torsion-free setting,  there is an overdetermined system of PDE,  the non-degenerate solutions of which correspond to compatible metrics.  \par  
We choose to rederive this metrisability equation from first principles here. 
All the arguments here are valid for all $n \ge 1$ unless otherwise specified.  Since we may have $n > 1$,  we do not have a natural object with which to raise and lower unprimed indices.  We do however introduce the notation,  that given a non-degenerate $\eta_{AB} \in \Gamma(\wedge^2S^*)$ we denote its \textit{algebraic dual} by $\eta^{AB} \in \Gamma(\wedge^2S)$ and vice versa.  That is $\eta_{AC}\eta^{BC} = \delta_{A}{}^{B}$.  Similarly for a top form $\omega_{AB...C} \in \Gamma(\wedge^{2n}S^*)$ we write $\omega^{AB...C} \in \Gamma(\wedge^{2n}S)$ for the unique section satisfying $\omega_{AB...C}\omega^{AB...C} = (2n)!$.  \par 

We need the following useful lemma:
\begin{lemma}\label{Weyl_LC_chracterisation}
Given a TFG,  the pair of Weyl connections $\nabla_{AA'}$ determined by $\omega_{AB...C} \in \Gamma(\wedge^{2n}S^*)$ and $\eta_{A'B'} \in \Gamma(\wedge^2 S'^*)$ satisfy
\begin{align}\label{metrisability_1}
\nabla_{AA'}\eta^{BC} = 2\delta_{A}{}^{[B} \mu^{C]}{}_{A'},  \\ 
\omega^{AB...CD} = \gamma \eta^{[AB}...\eta^{CD]} \label{metrisability_2}
\end{align}
for some skew $\eta^{BC} \in \Gamma(\wedge^2 S)$,  section $\mu^{AA'} \in \Gamma(TX[-1])$,   and constant $\gamma$  if and only if $\mu^{AA'} = 0$ and $\nabla_{AA'}$ induces the Levi-Civita connection for $g = \eta_{AB}\eta_{A'B'}$.  
\end{lemma}
\begin{proof}
If $\nabla_{AA'}$ is the Levi-Civita connection then 
\begin{align}
\nabla_{AA'}g_{BB'CC'} = (\nabla_{AA'}\eta_{BC})\eta_{B'C'} = 0,
\end{align}
so (\ref{metrisability_1}) is satisfied with $\mu^{AA'} = 0$.  Since $\eta_{AB}$ is parallel,  so must its $n$-fold wedge product be,  which implies it is $\omega_{AB...C}$ up to a constant.  \par 
Conversely (\ref{metrisability_2}) implies
\begin{align}
 \nabla_{AA'}(\eta^{[BC}\eta^{DE}...\eta^{FG]})=0.  
\end{align}
Combined with (\ref{metrisability_1}) it yields
\begin{align}
\delta_{A}{}^{[C}\mu^{B}{}_{A'}\eta^{DE}...\eta^{FG]} = 0.  
\end{align}
Contracting with $\eta_{[CB}...\eta_{FG]}$ we see $\mu^{AA'} = 0$ and so $\nabla_{AA'}$ induces the Levi-Civita connection.  
\end{proof}
\begin{prop}\label{can_fix_iso}
Suppose that the connections $\nabla_{AA'}$ determined by $\omega_{AB...C} \in \Gamma(\wedge^{2n}S^*)$ and $\eta_{A'B'} \in \Gamma(\wedge^2 S'^*)$ induce the Levi-Civita connection for $g = \eta_{AB}\eta_{A'B'}$.  Then so do the connections determined by $\Omega\omega_{AB...CD}$ and $\Omega^{-1/n}\eta_{A'B'}$. 
\end{prop}
\begin{proof}
We use Proposition \ref{general_change} with $\tilde{\Omega} = \Omega^{-1/n}$.  It gives 
\begin{align}
\hat{\nabla}_{AA'}\mu^{B} = \nabla_{AA'}\mu^{B} + \frac{1}{2n}\Upsilon_{AA'}\mu^{B}.  
\end{align}
From this we calculate $\hat{\nabla}_{AA'}\hat{\eta}_{BC} = 0$,  where $\hat{\eta}_{BC} = \Omega^{1/n} \eta_{BC}$.  
\end{proof}
Suppose we fix the auxiliary data of an isomorphism $\psi:  \wedge^2 S'^* \to \wedge^{2n} S^* $.  %In the presence of such an isomorphism,  given any algebraically compatible $g_{AA'BB'}$ there exists a factorisation 
%$g_{AA'BB'} = \eta _{AB}\eta_{A'B'}$
%such that $\psi(\eta_{A'B'}) = \eta_{[AB}...\eta_{CD]}$. 
What Proposition \ref{can_fix_iso} says is that in order to find Weyl connections which induce a Levi-Civita connection,  we do not need to vary the scales independently.  If there exist scales inducing a Levi-Civita connection for a compatible metric,  then there exist such scales related by $\psi$.  So in order to find all compatible metrics for a TFG,  we may as well fix a $\psi$,  and we need only consider transformations given by formulae (\ref{change_of_scales1}) and (\ref{change_of_scales2}) given in \cite{BE91}.  
\begin{thm}[Grassmannian metrisability equation]
\index{metrisability equation}
Given a TFG,  let $\nabla_{AA'}$ be any pair of Weyl connections.  There exists a compatible metric for the almost-Grassmannian structure if and only if there exists a non-degenerate $\eta^{AB} \in \Gamma(\wedge^2 S)$ and a $\mu^{AA'} \in \Gamma(TX[-1])$ such that \index{torsion-free Grassmannian structure!metrisability of}
\begin{align}\label{Grassmannian_metrisability}
\nabla_{AA'} \eta^{BC} = 2\delta_{A}{}^{[B} \mu^{C]}{}_{A'}. 
\end{align}
\end{thm}
\begin{proof} 
Suppose $\nabla_{AA'}$ correspond to scales $\eta_{A'B'}$ and $\omega_{AB...C}$.  Define $\psi: \wedge^2S'^* \to \wedge^{2n}S^*$ by $\psi(\eta_{A'B'}) = \omega_{AB...CD}$.  Suppose there exists $\mu^{AA'}$ satisfying (\textup{\ref{Grassmannian_metrisability}}).   From (\textup{\ref{change_of_scales1}}),  the condition (\textup{\ref{Grassmannian_metrisability}}) is invariant (after redefining $\mu^{AA'}$) under change of scales related by $\psi$.  Pick $\hat{\eta}_{A'B'}$ so $\psi(\hat{\eta}_{A'B'}) = \eta_{[AB}...
\eta_{CD]}$ and then the hypotheses of Lemma \textup{\ref{Weyl_LC_chracterisation}} are satisfied for $\hat{\nabla}_{AA'}$ to be the Levi-Civita connection for $\eta_{AB}\hat{\eta}_{A'B'}$.  Conversely if there exist Weyl connections $\hat{\nabla}_{AA'}$ which induce a Levi-Civita connection for $\eta_{AB}\hat{\eta}_{A'B'}$,  we know from Proposition \textup{\ref{can_fix_iso}} that without loss of generality we may assume the corresponding scales are related by $\psi$.  Then the formula (\textup{\ref{change_of_scales1}}) implies (\textup{\ref{Grassmannian_metrisability}}) is satisfied for $\nabla_{AA'}$ with $\mu^{A}{}_{A'} = \Upsilon_{BA'}\eta^{AB}$.  
\end{proof}
Note the equation is automatically satisfied for non-vanishing $\eta^{AB}$ when $n = 1$.  \par 
The equation (\ref{Grassmannian_metrisability}) is an overdetermined PDE that has been studied before (see \cite{F16,  HSSS12,  G25}) using the abstract machinery of parabolic geometry.   It defines a first Bernstein-Gelfand-Gelfand (BGG) operator\index{Bernstein-Gelfand-Gelfand operator} in almost-Grassmannian geometry.  
We may think of (\ref{Grassmannian_metrisability}) as an almost-Grassmannian analogue to the metrisability equation in projective geometry.  Solutions to that equation correspond to metrics for which the natural family of connections associated to the projective structure contains the Levi-Civita connection,  and it has been extensively studied both within the framework of parabolic geometry and much earlier: See for example \cite{EM07,  BDE09, DE16}.  \par 
\section{Prolongation of the metrisability equation and obstructions to metrisability}\label{tractor_obstructions}
Via \textit{invariant prolongation} \index{invariant prolongation}\index{metrisability equation!prolongation of} one may relate solutions to (\ref{Grassmannian_metrisability}) to sections of a vector bundle $\mathcal{A}$ that are parallel for a connection.  This prolongation was studied in \cite{HSSS12} Section 5.2.3 and \cite{HSSS12a}.  Full calculations in the notation we have adopted here may be found in \cite{Moy26}.  
The conclusion is the following:  We define a vector bundle that in the presence of a scale splits
\begin{align}\label{composition_series}
\mathcal{A} = \wedge^2S \oplus (S'{}^* \otimes S) \oplus \underbrace{\wedge^2S'{}^*}_{=\mathcal{O}[-1]}.
\end{align}
It is equipped with connection $D_{AA'}$ where:
\begin{align}
D_{AA'} \begin{bmatrix} \eta^{BC} \\ \mu_{B'}{}^{B} \\ \rho \end{bmatrix} = \begin{bmatrix} \nabla_{AA'}\eta^{BC} - 2\delta_{A}^{[B}\mu_{A'}{}^{C]} \\ \nabla_{AA'}\mu_{B'}{}^{B} - P_{AA'CB'}\eta^{BC} + \delta_{A}^{B}\epsilon_{A'B'}\rho \\ \nabla_{AA'}\rho + P_{AA'B}{}^{B'}\mu_{B'}{}^{B} \end{bmatrix}.  
\end{align}
$D_{AA'}$ is constructed in such a way that,  for $n > 1$,  solutions to (\ref{Grassmannian_metrisability}) correspond to parallel sections.  \par The connection still makes sense in the $n = 1$ case,  but now parallel sections correspond $\eta^{AB}$,  together with $\mu_{A'}{}^{A}$ defined by (\ref{Grassmannian_metrisability}) satisfying the additional condition coming from the vanishing of the middle slot for some $\rho$.  In this case,  the prolongation bundle and connection are equivalent to the conformal standard tractor bundle of \cite{BEG94} rewritten in spinor indices.   \par 
The bundle and connection are invariantly defined if the decomposition (\ref{composition_series}) depends on change in scale $\hat{\eta}_{A'B'} = \Omega \eta_{A'B'}$,  $\hat{\omega}_{AB...C} = \Omega \omega_{AB...C}$ in the following manner:
\begin{align}
\begin{bmatrix}
\hat{\eta}^{AB} \\
\hat{\mu}^{A}{}_{A'}  \\
\hat{\rho} 
\end{bmatrix} = 
\begin{bmatrix}
\eta^{AB} \\
\mu^{A}{}_{A'}  - \Upsilon_{BA'}\eta^{AB} \\
\rho - \Upsilon_{AA'}\mu^{AA'} + \frac{1}{2}\Upsilon_{AA'}\Upsilon_{B}{}^{A'}\eta^{AB}
\end{bmatrix}.  
\end{align}
Note that although (\ref{composition_series}) is not an invariantly defined direct sum decomposition,  the bundle does have an invariantly defined two-step filtration.  $\mathcal{A}$ may be realised as a subbundle of the second jet bundle $J^2(\wedge^2S)$ cut out by the equation (\ref{Grassmannian_metrisability}).  \par 
We in fact have an identification $\mathcal{A} = \wedge^2 \mathcal{T}$ where $\mathcal{T}$ is the \textit{standard tractor bundle}\index{tractor!bundle!standard} that in the presence of a scale splits $\mathcal{T} = S \oplus S'^*$ such that decomposition changes according to 
\begin{align}
\begin{bmatrix}
\hat{\sigma}^{A} \\
\hat{\mu}_{A'} 
\end{bmatrix} 
=
\begin{bmatrix}
\sigma^{A} \\
\mu_{A'} - \Upsilon_{AA'}\sigma^{A}  
\end{bmatrix} 
\end{align}
$D_{AA'}$ on $\mathcal{A}$ is induced by the connection 
\begin{align}
D_{AA'} \begin{bmatrix}
{\sigma}^{B} \\
{\mu}_{B'} 
\end{bmatrix} = 
\begin{bmatrix}
\nabla_{AA'}\sigma^{B} + \delta_{A}{}^{B}\mu_{A'} \\
\nabla_{AA'}\mu_{B'} - P_{AA'BB'}\sigma^{B} 
\end{bmatrix}
\end{align}
and the identification $\mathcal{A} = \wedge^2 \mathcal{T}$. \par 
The connection $D_{AA'}$ is the normalised tractor connection induced by the Cartan connection\index{Cartan connection} arising from the general theory of parabolic geometry.  For a general almost-Grassmannian structure (so possibly in the presence of torsion),  the prolongation of (\ref{Grassmannian_metrisability}) may not lead to the Cartan connection (see \cite{HSSS12} for the general expression).  \par 
In older terminology,  the bundle $\mathcal{T}$ equipped with the connection $D_{AA'}$ generalises the \textit{local twistor} bundle\index{twistor!local} and connection of four-dimensional conformal geometry.  See also the exposition of how this bundle is realised via the Ward correspondence in \cite{Lam23}.   \par 
The curvature of $D_{AA'}$ on $\mathcal{T}$ for a TFG is the homomorphism $\mathcal{T} \to \wedge^2 T^*X \otimes \mathcal{T}$ given by\index{tractor!connection!curvature of}
\begin{align}\label{tractor_curvature}
\begin{bmatrix}
{\sigma}^{A} \\
{\mu}_{A'} 
\end{bmatrix} \mapsto \begin{bmatrix}
\Psi_{ABD}{}^{C}\sigma^{D}\epsilon_{A'B'} \\
-(\nabla_{AA'}P_{BB'DC'} - \nabla_{BB'}P_{AA'DC'})\sigma^{D} 
\end{bmatrix}. 
\end{align}
%The following are helpful:
%\begin{prop}[Bianchi identities for TFGs]\label{TFG_Bianchi}
%\index{Bianchi identity!for TFGs}
%The following hold for general Weyl connections on a TFG
%\begin{align}
%\nabla_{[A}{}^{A'}\Phi_{B]CA'B'} + 3\nabla_{[A|B'}\Lambda_{|B]C} &= 0,  \label{Bianchi1} \\
%\nabla_{[A|A'}\Lambda_{|BC]} &= 0,  \label{Bianchi2}\\
%\nabla_{DA'}\Psi_{ABC}{}^{D} + (2n-1)\nabla_{A}{}^{B'}\Phi_{BCA'B'} + 2(2n-1)\nabla_{(B|A'}\Lambda_{|C)A} &= 0,  \label{Bianchi3}\\
%\nabla_{DA'}\Psi_{ABC}{}^{D} + (2n-1)\nabla_{(A}{}^{B'}P_{B)B'CA'} &= 0, \label{Bianchi4} \\ 
%\nabla_{[A|}{}^{(A'}P_{B]}{}^{B')}{}_{CC'} &= 0.  \label{Bianchi5}
%\end{align}
%\end{prop}
%It is explained how to derive these in Appendix \ref{bianchi_appen}.  
From the Bianchi identity (given in \cite{BEG94} or following from calculations in Appendix \ref{bianchi_appen})
\begin{align}
\nabla_{AA'}\Psi_{BCD}{}^{A} + (2n-1)\nabla^{B'}{}_{(B}P_{C)B'DA'}
\end{align}
and the Leibniz rule,  we see the curvature $F_{AA'BB'} \in \Gamma(\wedge^2T^*X \otimes \End(\mathcal{A}))$ (suppressing tractor indices) of $D_{AA'}$ on $\mathcal{A}$ is given by
\begin{align}
F_{AA'BB'}
\begin{bmatrix} \eta^{CD} \\ \mu_{C'}{}^{C} \\ \rho \end{bmatrix}  = 
 \begin{bmatrix} -2\epsilon_{A'B'}\Psi_{ABE}{}^{[C}\eta^{D]E}
 \\ \epsilon_{A'B'}\Psi_{ABD}{}^{C}\mu_{C'}{}^{D} - \frac{1}{(2n-1)}\epsilon_{A'B'}\nabla_{EC'}\Psi_{ABD}{}^{E}\eta^{CD} \\
-\frac{1}{(2n-1)}\epsilon_{A'B'}\nabla_{DC'}\Psi_{ABC}{}^{D}\mu^{CC'}
 \end{bmatrix},  
\end{align}
which vanishes if and only if $\Psi_{ABC}{}^{D} = 0$ and the TFG is locally isomorphic to the flat model.  \par 
From this curvature formula we obtain necessary algebraic conditions on sections of $\mathcal{A}$ to be parallel for $D_{AA'}$.  Since a compatible metric corresponds to a parallel section with non-degenerate $\eta^{AB}$,  we may readily use these conditions to derive algebraic obstructions to the existence of a compatible metric.  The invariance of the tractor connection implies the formulae for such obstructions will be valid in an arbitrary scale.  
For example:
\begin{prop}\label{easy_obstruct}
If there exists a compatible metric,  then for $k \in \mathbb{N}$ we have
\begin{align}
&\Psi_{A_1B_1C_1}{}^{C_2}\Psi_{A_2B_2C_2}{}^{C_3}...\Psi_{A_kB_kC_k}{}^{C_1} \nonumber \\ + (-1)^{k+1}&\Psi_{A_kB_kC_1}{}^{C_2}\Psi_{A_{k-1}B_{k-1}C_2}{}^{C_3}...\Psi_{A_1B_1C_k}{}^{C_1} = 0. 
\end{align}
\end{prop}
\begin{proof}
The $k = 1,2$ cases are trivial identities and not obstructions.  For $k \ge 3$,  we will lower and raise indices with a non-degenerate $\eta^{AB}$.  Then the condition $\Psi_{ABE}{}^{[C}\eta^{D]E} = 0$ is the condition $\Psi_{AB(CD)} = \Psi_{ABCD}$.  We may therefore write
\begin{align}
\Psi_{A_1B_1C_1}{}^{C_2}\Psi_{A_2B_2C_2}{}^{C_3}...\Psi_{A_kB_kC_k}{}^{C_1} 
 &= \Psi_{A_1B_1}{}^{C_2}{}_{C_1}\Psi_{A_2B_2}{}^{C_3}{}_{C_2}...\Psi_{A_kB_k}{}^{C_1}{}_{C_k}  \nonumber \\
 &= (-1)^{k}\Psi_{A_1B_1C_2}{}^{C_1}\Psi_{A_2B_2C_3}{}^{C_2}...\Psi_{A_kB_kC_1}{}^{C_k},  
\end{align}
 where in the last line we have used the fact we may change the vertical position of a pair of contracted indices at the cost of a sign.  
\end{proof}
It is the Razmyslov-Procesi\index{Razmyslov-Procesi theorem} \cite{R74,  Pr76} theorem that given $m$ matrices of dimension $l \times l$,  all polynomial invariants of the set of matrices may be expressed in terms of traces of degree less than or equal to $l^2$.  Pick a basis so that $\Psi_{ABC}{}^{D}$ defines a tuple of $2n(2n+1)/2$ matrices (think of the last two indices as components of the matrices).  The theorem tells us the above formula stops producing independent obstructions when $k > 4n^2$.  \par 
\begin{exmp}
Let us construct an example for which the $k = 3$ obstruction does not vanish.  Take constant vectors $K^{a},  J_{a}$ with $K^{a}J_{a} \ne 0$ and some integer $\mu \ge 3$.  Write $J \cdot \theta = J_{a}\theta^{a}$.  Consider
 \textup{(\ref{hyper_hermitian_example})} with 
 \begin{align}
 W^{a} = K^{a}(J \cdot \theta)^\mu.  
 \end{align}
This solves \textup{(\ref{hyper_hermitian_right_flat})} and we have,  normalising so $K^{a}J_{a} = 1$, 
\begin{align}
\Psi_{abc}{}^{d} = \gamma(\theta)(J_{a}J_{b}J_{c}K^{d} - \frac{3}{2n+2}J_{(a}J_{b}\delta_{c)}{}^{d})
\end{align}
for $\gamma(\theta) = \mu(\mu-1)(\mu-2)(J \cdot \theta)^{(\mu-3)}$.  It is straightforward to compute
\begin{align}
K^{a}K^{b}K^{e}K^{f}K^{h}K^{i}\Psi_{abc}{}^{d}\Psi_{efd}{}^{g}\Psi_{hig}{}^{c} = \gamma(\theta)^3 \cdot \frac{2n(2n-1)(2n-2)}{(2n+2)^3},
\end{align}
which shows the obstruction of Proposition \textup{\ref{easy_obstruct}} does not vanish unless $n=1$.  \par 
On the other hand taking $W^{a}$ as in \textup{(\ref{simple_soln})},  the obstructions trivially vanish.  For this example there is a non-degenerate solution $\eta^{de}$ to $\Psi_{abe}{}^{[c}\eta^{d]e} = 0$ if and only if $T_{abc} = J_{(a}J_bJ_{c)}$ for some constant vector $J_{a}$.  This shows vanishing of the obstructions above is not sufficient for the existence of a non-degenerate solution.  
\end{exmp}
\section{Submaximally symmetric Grassmannian structures}\label{submax}
Kruglikov and The \cite{KT} exhibited a symmetry gap for almost-Grassmannian structures (as a special case of their general results for parabolic geometries).  That is,  they showed the following example has the maximal number ($4n^2 + 5$) of linearly independent infinitesimal symmetries in the non-flat setting: 
\begin{exmp}[Kruglikov-The model]
\label{Kruglikov_The}
We refer back to Example \textup{\ref{hyper_hermitian_example}}.  
Let $J_a$ and $K^a$ be non-zero constant vectors $J_aK^a=0$
and write $J\cdot\theta:=J_a\theta^a$.  Set
\begin{align}
W^a=-\frac{1}{2}K^a(J\cdot\theta)^3.
\label{Kruglikov_The_W}
\end{align}
Then 
\begin{align}
U_a
&=
\frac{\partial}{\partial z^a}
+
\frac{3}{2}J_a(J\cdot\theta)^2
K^b\frac{\partial}{\partial\theta^b},
\\
V_a
&=
\frac{\partial}{\partial\theta^a}.
\end{align}
Clearly \textup{(\ref{hyper_hermitian_right_flat})} holds.  
The associated Weyl connections are defined by
$$
\Gamma_{aa'b}{}^c
=
-3(J\mathbin{\cdot}\theta)
J_aJ_bK^c\delta_{a'}{}^{1'},
\qquad
\Gamma_{aa'b'}{}^{c'}=0.
$$
So the scale is Ricci-flat and the remaining curvature is the invariant part:
\begin{align}
\Psi_{abc}{}^d
=
-3J_aJ_bJ_cK^d.
\label{Kruglikov_The_curvature}
\end{align}
\end{exmp}
The generators of the symmetries are provided explicitly in \cite{KT}.  A later result of The \cite{The} shows this submaximal model is unique.  
\begin{prop}[Compatible metrics of the submaximally symmetric model]
Suppose $n > 1$.  The submaximally symmetric model Example \textup{\ref{Kruglikov_The}} admits at least $2n^2 - 3n + 4$ linearly independent compatible metrics.  Explicitly 
\begin{align}
\eta^{ab} = \sigma^{ab} + (\beta z^{[a} + \gamma \theta^{[a})K^{b]}
\end{align}
for some constants $\beta,  \gamma,  \sigma^{ab} = -\sigma^{ba}$,  
where $\sigma^{ab}$ must satisfy $\sigma^{ab}J_{b} = \alpha K^{a}$ for some constant $\alpha$.  
\end{prop}
\begin{proof}
In the given Weyl connections we have
\begin{align}
\nabla_{aa'}\eta^{bc} = e_{aa'}(\eta^{bc}) + 6(J \cdot \theta)J_{a}J_{d}K^{[b}\eta^{c]d}\delta_{a'}{}^{1'},
\end{align}
where $e_{a0'} = V_{a}$ and  $e_{a1'} = U_{a}$.  
So the metrisability equation is 
\begin{align}
e_{aa'}(\eta^{bc}) - \frac{2}{2n-1}e_{da'}(\eta^{d[c})\delta_{a}{}^{b]} + 6(J \cdot \theta)J_{a}J_{d}K^{[b}\eta^{c]d}\delta_{a'}{}^{1'} = 0.  
\end{align}
There is an obvious solution,  which is to take $\eta^{ab} = \sigma^{ab}$ constant and with the right structure for the last term to vanish.  This is precisely $\sigma^{ab}J_{b} = \alpha K^{a}$ for some constant $\alpha$,  which provides $2n-2$ independent constraints on what are otherwise $2n^2-n$ constants which we are free to choose.  This yields $2n^2 - 3n + 2$ compatible metrics.  To find more,  note that setting $\eta^{ab} = X^{[a}K^{b]}$ for some (potentially non-constant) $X^{a}$ again kills the last term.  The equations reduce to 
\begin{align}
\bigg(\frac{\partial X^{[b}}{\partial \theta^{a}}K^{c]}\bigg)_\circ = 0 ,  \quad a' &= 0',  \\
\bigg(\frac{\partial X^{[b}}{\partial z^{a}}K^{c]} + \frac{3}{2}J_{a}(J \cdot \theta)^2K^{d}\frac{\partial X^{[b}}{\partial \theta^{d}}K^{c]}  \bigg)_\circ = 0 ,  \quad a' &= 1'.  
\end{align}
The general solution to the first equation is $X^{a} = \gamma \theta^{a} + f^{a}(z)$,  and note that given this,  the second term in the second equation vanishes.  So the general solution for this ansatz is $X^{a} = \beta z^{a} + \gamma \theta^{a}$ for some constants $\beta,  \gamma$.  This provides the two additional degrees of freedom.  
\end{proof}
The metrics with $\beta = \gamma = 0$ are hyper-K\"ahler,  whereas the metrics in the family generically have non-zero scalar curvature.  The hyper-K\"ahler examples are the product of an ASD Ricci-flat pp-wave with a flat factor.  \par 
Of course,  an infinitesimal symmetry of the Grassmannian structure may not preserve the extra structure of a chosen compatible metric,  so it may be interesting to describe how they act on the space of compatible metrics.  \par 
In the flat model,  there are $\text{rank} \mathcal{A} = (2n + 1)(n + 1)$ linearly independent compatible metrics.  Here we bound the next allowed number from above and it turns out this bound is saturated by the last example. 
\begin{prop}[Bound on compatible metrics]
Let $n > 1$.  For a non-flat TFG there are at most $2n^2 - 3n + 4$ linearly independent compatible metrics.  
\end{prop}
\begin{proof}
Let us suppose there exists a compatible metric and then work in the corresponding Einstein scale.  The (tractor) derivative of the non-vanishing (that is,  the anti-self-dual) $F_{AB} := F_{AA'B}{}^{A'}$ part of the tractor curvature is:
\begin{align}
D_{EE'}F_{AB}
\begin{bmatrix}
\eta^{CD}\\
\mu_{C'}{}^C\\
\rho
\end{bmatrix}
=
\begin{bmatrix}
-2\nabla_{EE'}\Psi_{ABF}{}^{[C}\eta^{D]F}
+2\Psi_{ABE}{}^{[C}\mu_{E'}{}^{D]}
\\[2mm]
\nabla_{EE'}\Psi_{ABF}{}^C\mu_{C'}{}^F
-\Psi_{ABE}{}^C\epsilon_{E'C'}\rho
+\Lambda_{EF}\epsilon_{E'C'}\Psi_{ABG}{}^F\eta^{CG}
\\[2mm]
-\Lambda_{EF}\Psi_{ABG}{}^F\mu_{E'}{}^G
\end{bmatrix}.
\end{align}
Necessary conditions on a parallel section of $\mathcal{A}$ for $D$ are
\begin{align}\label{conds}
F_{AB}\begin{bmatrix}
\eta^{CD}\\
\mu_{C'}{}^C\\
\rho
\end{bmatrix} = 0,  \quad 
D_{EE'}F_{AB}
\begin{bmatrix}
\eta^{CD}\\
\mu_{C'}{}^C\\
\rho
\end{bmatrix}
= 0.  
\end{align}
Let $\mathcal{V}$ be the vector space of solutions to (\ref{conds}) at a point.  We have a map $\phi_1: \mathcal{V} \to \wedge^2 S$ given by projection in the first tractor slot.  
Now,  rank-nullity says:
\begin{align}
\dim \mathcal{V} = \dim \im \phi_1 + \dim \ \mathcal{V}|_{\eta=0} 
\end{align}
and we can similarly find the dimension of $ \mathcal{V}|_{\eta = 0} $ by projecting $\phi_2: \mathcal{V}|_{\eta = 0} \to S'^* \otimes S$.  Writing $\phi_3: \mathcal{V}|_{\eta,  \mu = 0} \to \mathcal{O}[-1]$ for the projection onto the last slot then 
\begin{align}
\dim \mathcal{V} = \dim \im \phi_1 +  \dim \im \phi_2 +  \dim \im \phi_3.  
\end{align}
We will now bound the $\dim \im \phi_i$ from above for each $i$.
Clearly $\dim \im \phi_3 = 0$.  \par $\dim \im \phi_2$ is bounded above by the dimension of the space of solutions $\mu_{A'}{}^{A}$ to 
\begin{align}
\Psi_{ABE}{}^{[C}\mu_{E'}{}^{D]} = 0.  
\end{align}
Pick $\lambda^{ABC}$ so that $\psi^{A} := \lambda^{BCD}\Psi_{BCD}{}^{A} \ne 0$ at the point.  Then,  the dimension of the solution space to the above equation is bounded above by the dimension of the solution space to $\psi^{[A}\mu_{C'}{}^{B]} = 0$,  which is two-dimensional.  Slightly more complicated is bounding the dimension of the solution space to 
\begin{align}
\Psi_{ABE}{}^{[C}\eta^{D]E} = 0.  
\end{align}
Pick $\lambda^{AB}$ so that $\psi_{A}{}^{B} := \lambda^{CD}\Psi_{CDA}{}^{B} \ne 0$ defines a trace-free endomorphism.  This implies there exists $\phi^{B}$ such that $\phi^{A}\psi_{A}{}^{[B} \phi^{C]} \ne 0$.  Consider the restriction of the map
\begin{align}\label{skew_map}
\eta^{AB} \mapsto \psi_{C}{}^{[A}\eta^{B]C} 
\end{align}
to simple tensors $\phi^{[A}\nu^{B]}$.  Post-composing this map with taking the wedge product with $\phi$ has image $\phi \wedge \psi(\phi) \wedge S$.  We therefore see the rank of (\ref{skew_map}) is at least $2n-2$ and therefore $\dim \im \phi_1 \le n(2n-1) - 2n + 2 = 2n^2 - 3n + 2$.  So we have $\dim \mathcal{V} \le 2n^2 - 3n + 4$.  
\end{proof}
\begin{cor}
Example \textup{\ref{Kruglikov_The}} admits the submaximally allowed number of linearly independent compatible metrics among TFGs.    
\end{cor}
%Infinitesimal symmetries of parabolic geometries correspond to (under some mild assumptions) sections of the adjoint tractor bundle parallel for the Cartan connection.  Generally,  it is not clear how one should relate parallel adjoint tractors with parallel sections of the bundle $\mathcal{A}$.  For example there are Ricci-flat ASD metrics with no conformal Killing vectors.  In the four-dimensional case we have seen there is a relation in the presence of a K\"ahler structure,  and it would be interesting to generalise this to higher dimensions.  
\section{Einstein scales in Grassmannian geometry}\label{AGEinstein}
In this section,  instead of considering compatible metrics,  we consider related objects which are scales satisfying a Grassmannian generalisation of the Einstein condition.  \par 
We define,  as per \cite{BE91}:
\begin{defn}[Einstein scale]
Given a TFG,  an \textit{Einstein scale} is a pair of Weyl connections $\nabla_{AA'}$ for which $\Phi_{ABA'B'} = 0$.  
\end{defn}
Suppose we have a parallel section of $D_{AA'}$ with non-degenerate $\eta^{AB}$.  Working in the Levi-Civita scale corresponding to $\omega_{AB...C}$ and $\eta_{A'B'}$ where $\nabla_{AA'}\eta^{BC} = 0$ we have
\begin{align}
P_{AA'CB'}\eta^{BC} = \delta_{A}{}^{B}\epsilon_{A'B'}\rho.  
\end{align}
The two irreducible components of this equation are
\begin{align}\label{Phi_killed}
\Phi_{ACA'B'}\eta^{BC} = 0 
\end{align}
and \index{torsion-free Grassmannian structure!metrisability of}
\begin{align}
\Lambda_{AB}\epsilon_{A'B'} = -\eta_{AB}\epsilon_{A'B'}\rho.
\end{align}
Since $\eta^{BC}$ is non-degenerate,  the first equation implies $\Phi_{ABA'B'} = 0$.  
This implies the Ricci tensor is totally trace and hence that $\eta_{AB}\eta_{A'B'}$ is an Einstein metric.
This is (in the complexified setting) the well-known result that a quaternion-K\"ahler metric is Einstein.  
There is a kind of converse,  noted in \cite{BE91},  as we will explain.  It relies on an analogue of the fact the scalar curvature is constant for an Einstein metric.  We prove the following in Appendix \ref{bianchi_appen}:
\begin{lemma}[Vacuum Bianchi identity I]\label{Vacuum Bianchi identity I}
In an Einstein scale on a TFG:
\begin{align}
 \nabla_{AA'}\Lambda_{BC} = 0. 
 \end{align} 
\end{lemma}
When $n = 1$,  an Einstein scale gives an Einstein metric in the conformal class.  
For general $n$,  the above fact implies that,  if $\Lambda_{AB}$ is non-degenerate,  the connections $\nabla_{AA'}$ induce the Levi-Civita connection for the metric $g_{AA'BB'} = \Lambda_{AB}\epsilon_{A'B'}$.   \par
Having illuminated the relationship between Einstein scales and compatible metrics along the lines of \cite{BE91},  we will now adapt the methods of \cite{S63,  KNT85},  as we have already done in the $n=1$ setting in \S \ref{KNT_repeat},  to find (given a generic condition) necessary and sufficient conditions for the existence of Einstein scales when $n \ge 1$.  \par 
Similarly to when considering the existence of compatible metrics,  it is safe to fix an isomorphism  $\psi:  \wedge^2 S'^* \to \wedge^{2n} S^*$ and consider scales related by $\psi$.  To see this,  note from Proposition \ref{general_change},  that the induced connection on $TX \cong S \otimes S'$ is preserved under change of scales $\Omega\omega_{AB...CD}$ and $\Omega^{-1/n}\eta_{A'B'}$.  The Einstein condition is the vanishing of the component of the curvature of this connection lying in an irreducible representation of $GL(2n, \mathbb{C}) \times GL(2,\mathbb{C})$,  which is clearly a condition independent of the choice of scales.  \par 
There is a generalisation of the Bach-vanishing condition for TFGs.  It can be obtained most easily via the Bianchi identity for the tractor connection $D_{AA'}$ on $\mathcal{T}$.  It is:
\begin{prop}[Tractor Bianchi identity]\label{Bach_vanishing}
The following identity holds on a TFG:
\begin{align}
\nabla_{[A|}{}^{(A'}\nabla_{E}{}^{B')}\Psi_{|B]CD}{}^{E} + (2n-1)\Phi_{[A|E}{}^{A'B'}\Psi_{|B]CD}{}^{E} = 0.  
\end{align}
\end{prop} 
We now introduce some non-degeneracy conditions for the invariant curvature:
\begin{defn}[$\wedge^2$-non-degeneracy]
\index{torsion-free Grassmannian structure!$\wedge^2$-non-degeneracy of}
Suppose that $\Psi_{ABC}{}^{D}$ defines an injective map $\odot^2 S^* \to \wedge^2 S^* \otimes \odot^2 S^{*}[-1]$ given by 
\begin{align}
\phi_{AB} \mapsto \phi_{E[A}\Psi_{B]CD}{}^{E},  
\end{align}
then we say the TFG is $\wedge^2$-non-degenerate.  
\end{defn}
When $n = 1$ this condition is equivalent to $J \ne 0$.  The terminology here is adopted because this condition is a natural analogue to the condition in \cite{GN06} with the same name.   We also will later use:
\begin{defn}[Weak non-degeneracy]
\index{torsion-free Grassmannian structure!weak non-degeneracy of}
Suppose that $\Psi_{ABC}{}^{D}$ defines an injective map  $S^* \to \odot^3 S^{*}[-1]$ given by 
\begin{align}
\phi_{A} \mapsto \phi_{D}\Psi_{ABC}{}^{D},  
\end{align}
then we say the TFG is weakly non-degenerate.  
\end{defn}
Note that $\wedge^2$-non-degeneracy implies weak non-degeneracy: If $\phi_{D}\Psi_{ABC}{}^{D} = 0$ then
 \begin{align}
\phi_{A}\phi_{B} \mapsto \phi_{E}\phi_{[A}\Psi_{B]CD}{}^{E} = 0.  
\end{align}
Note that weak non-degeneracy just implies ``not type-N" in four dimensions.  \par 
We prove,  in Appendix \ref{bianchi_appen} an analogue of the ``vacuum Bianchi identity":
\begin{lemma}[Vacuum Bianchi identity II]\label{Vacuum Bianchi identity II}
The following holds on a TFG in an Einstein scale: \index{Bianchi identity!vacuum}
\begin{align}\label{vbi}
\nabla_{[A|A'}\Psi_{|B]CD}{}^{E}  = 0.  
\end{align}
\end{lemma}
From applying the formulae (\ref{change_of_scales1}) and (\ref{change_of_scales3}) to (\ref{vbi}),  we see that a necessary condition for the existence of an Einstein scale is the existence of a closed $1$-form $\Upsilon_{AA'}$ such that 
\begin{align}\label{para-C-space}
\nabla_{[A|A'}\Psi_{|B]CD}{}^{E} + \Upsilon_{FA'}\delta_{[A}{}^{E}\Psi_{B]CD}{}^{F} = 0.  
\end{align}
Proposition \ref{C-space equivalent to conformal Einstein generically} generalises:
\begin{prop}
A $\wedge^2$-non-degenerate TFG has an Einstein scale if and only if there exists a closed $1$-form $\Upsilon_{AA'} \in \Gamma(T^*X)$ satisfying \textup{(\ref{para-C-space})}.   
\end{prop}
\begin{proof}
Trace and differentiate (\ref{para-C-space}):
\begin{align}
\nabla_{F}{}^{A'}\big(\nabla_{E}{}^{B'}\Psi_{BCD}{}^{E} + (2n-1)\Upsilon_{E}{}^{B'}\Psi_{BCD}{}^{E}\big)=0.  
\end{align}
Using Proposition \ref{Bach_vanishing} and then (\ref{para-C-space}) again,  one obtains
\begin{align}
\Big(\Phi_{[A|F}{}^{A'B'} - \nabla_{[A|}{}^{(A'}\Upsilon_{F}{}^{B')} + \Upsilon_{[A|}{}^{(A'}\Upsilon_{F}{}^{B')}\Big)\Psi_{|B]CD}{}^{F} = 0,
\end{align}
and using the genericity condition this implies
\begin{align}
\Phi_{AB}{}^{A'B'} - \nabla_{(A}{}^{(A'}\Upsilon_{B)}{}^{B')} + \Upsilon_{(A}{}^{(A'}\Upsilon_{B)}{}^{B')} = 0. 
\end{align}
From (\ref{schouten_change}) we see this is precisely the relation for $\Phi_{AB}{}^{A'B'}$ to vanish in the scale related by the factor $\Omega$ such that $\Upsilon_{AA'} = \nabla_{AA'} \log \Omega$.  
\end{proof}
The trace of (\ref{para-C-space}) is:
\begin{align}\label{para-C-space-trace}
\nabla_{AA'}\Psi_{BCD}{}^{A} + (2n-1)\Upsilon_{AA'}\Psi_{BCD}{}^{A}=0.  
\end{align}
Weak non-degeneracy is the condition that there exists a (non-unique) left inverse $\Xi_{A}{}^{BCD}$ such that 
\begin{align}
\Xi_{A}{}^{CDE}\Psi_{CDE}{}^{B} = \delta_{A}{}^{B}.  
\end{align}
In this case we may solve for a unique algebraic solution $U_{AA'}$ to (\ref{para-C-space-trace}),  if it exists,  as 
\begin{align}
U_{AA'} = -\frac{1}{2n-1}\Xi_{A}{}^{BCD}\nabla_{EA'}\Psi_{BCD}{}^{E}.  
\end{align}
Substituting this into (\ref{para-C-space}) gives necessary and sufficient conditions for an algebraic solution of (\ref{para-C-space}) analogous to the vanishing of the invariant (\ref{BEinvariant}). \par   
The next natural question is whether there are integrability conditions (presumably using $\wedge^2$-non-degeneracy) on an algebraic solution of (\ref{para-C-space}) that imply it is closed.  We leave this for future work,  and one may speculate the answer is related to an appropriate generalisation of Einstein-Weyl structure to this setting,  and the BGG sequences that underlie the corresponding fact in the ASD case,  as explained in \cite{Calderbank}.  
\section{Conclusion}
We have considered some problems,  old and new,  relating to the ASD-Einstein equations and their generalisation in Grassmannian geometry.  \par An interesting application of the work in this article would be to attempt to apply it to the $(2n,2)$-Grassmannian structures that are a feature of (complexified) moduli spaces of monopoles \cite{BS13, FH24, FH26}.  \par 
There is a construction \cite{CS07,  Z25} of $(2n,2)$-almost-Grassmannian structures from path geometries of dimension $2n$ (this construction is different from the usual twistor correspondence).  It would be interesting to see what data on a path geometry corresponds to a compatible metric on the Grassmannian geometry produced by this construction.

\appendix

\section{Adapted connections on (almost-)Grassmannian structures}
\begin{prop}[Change of Weyl connections]\label{general_change}\index{Weyl connections!formula for change of scale}
The connections of Proposition \textup{\ref{fundamental_theorem}} corresponding to $\hat{\omega}_{AB...C} = \Omega \omega_{AB...C}$ and
$\hat{\eta}_{A'B'} = \tilde{\Omega} \eta_{A'B'}$ are given by 
\begin{align}
\hat{\nabla}_{AA'}\mu^B \hspace{1mm} = \nabla_{AA'}\mu^B + &\frac{1}{2n+2}\Upsilon_{AA'}\mu^{B} + \frac{2}{2n+2}\Upsilon_{CA'}\delta_{A}{}^{B}\mu^{C} \nonumber \\
 - &\frac{1}{2n+2}\tilde{\Upsilon}_{AA'}\mu^{B} + \frac{2n}{2n+2}\tilde{\Upsilon}_{CA'}\delta_{A}{}^{B}\mu^{C}, \label{full_change_1} \\
\hat{\nabla}_{AA'}\mu^{B'} = \nabla_{AA'}\mu^{B'} - &\frac{1}{2n+2}\Upsilon_{AA'}\mu^{B'} + \frac{2}{2n+2}\Upsilon_{AC'}\delta_{A'}{}^{B'}\mu^{C'} \nonumber \\
+&\frac{1}{2n+2}\tilde{\Upsilon}_{AA'}\mu^{B'} + \frac{2n}{2n+2}\tilde{\Upsilon}_{AC'}\delta_{A'}{}^{B'}\mu^{C'},  
 \label{full_change_2}
\end{align}
where $\Upsilon_{AA'} = \nabla_{AA'} \log \Omega$,  and $\tilde{\Upsilon}_{AA'} = \nabla_{AA'} \log \tilde{\Omega}$.  
\end{prop}
\begin{proof}
From the proof of Theorem 2.4 in \cite{BE91},  the totally trace-free torsion is preserved if and only if the pair of connections (denoted $\hat{\nabla}_{AA'}$) differ from $\nabla_{AA'}$ by:
\begin{align}
\hat{\nabla}_{AA'}\mu^{C} &= \nabla_{AA'}\mu^{C}  + U_{(A|A'}\delta_{|B)}{}^{C}\mu^{B} + V_{[A|A'}\delta_{|B]}{}^{C}\mu^{B} \\
\hat{\nabla}_{AA'}\mu^{C'} &= \nabla_{AA'}\mu^{C'} - V_{A(A'}\delta_{B')}{}^{C'}\mu^{B'} - U_{A[A'}\delta_{B']}{}^{C'}\mu^{B'}, 
\end{align}
for $1$-forms $U_{AA'},  V_{AA'}$.  \par
Let us first compute the change of connections with $\tilde{\Omega} = 1$ and $\Omega$ arbitrary.  In order for the connection on $S'$ preserve $\eta_{A'B'}$,  it is easy to check that we require $U_{AA'} = -3V_{AA'}$.  Next,  let $\omega^{C_1...C_{2n}} \in \Gamma(\wedge^{2n}S)$ satisfy $\omega^{C_1...C_{2n}}\omega_{C_1...C_{2n}} = 2n!$.  Now compute
\begin{align}
&\hat{\nabla}_{AA'}(\Omega^{-1}\omega^{C_1...C_{2n}})\omega_{C_1...C_{2n}} \nonumber \\ =& \ 2n! \nabla_{AA'}\Omega^{-1} + 2n\Omega^{-1}( -3V_{(A|A'}\delta_{|B)}{}^{C_1} +V_{[A|A'}\delta_{|B]}{}^{C_1})\omega^{BC_2...C_{2n}}\omega_{C_1...C_{2n}}
\end{align}
From $\omega^{A...C_{2n}}\omega_{B...C_{2n}} = (2n-1)!\delta_{B}{}^{A}$ we see that $\hat{\omega}_{AB...C} = \Omega \omega_{AB...C}$ is preserved if and only if
\begin{align}
V_{AA'} = -\frac{1}{2n+2}\Upsilon_{AA'}.  
\end{align}
On the other hand we consider arbitrary $\tilde{\Omega}$ and  $\Omega = 1$.  In order to preserve the section $\omega_{C_1...C_{2n}}$ we require 
\begin{align}
( U_{(A|A'}\delta_{|B)}{}^{C_1} +V_{[A|A'}\delta_{|B]}{}^{C_1})\omega^{BC_2...C_{2n}}\omega_{C_1...C_{2n}}=0,
\end{align}
which yields
\begin{align}
V_{AA'} = -\frac{2n+1}{2n-1}U_{AA'}.  
\end{align}
We now insist that $\tilde{\Omega}\eta_{A'B'}$ is parallel.  Explicitly,  we require the vanishing of 
\begin{align}
&\hat{\nabla}_{AA'}(\tilde{\Omega}^{-1}\eta^{B'C'}) \\ = \ &(\nabla_{AA'}\tilde{\Omega}^{-1})\eta^{B'C'} - 2\tilde{\Omega}^{-1} \frac{2n+1}{2n-1}U_{A(A'}\delta_{D')}{}^{[B'}\eta^{C']D'} + 2\tilde{\Omega}^{-1}U_{A[A'}\delta_{D']}{}^{[B'}\eta^{C']D'},  \nonumber
\end{align}
and we solve 
\begin{align}
U_{AA'} = \frac{2n-1}{2n+2}\tilde{\Upsilon}_{AA'}.  
\end{align}
Performing these rescalings in either order yields the formulae (\ref{full_change_1}) and (\ref{full_change_2}).  
\end{proof} 
These clearly reduce to (\ref{change_of_scales1}) and (\ref{change_of_scales2}) if $\tilde{\Omega} = \Omega$.  

\section{Bianchi identities}\label{bianchi_appen}\index{torsion-free Grassmannian structure!Bianchi identities for}
\begin{proof}[Proof of Lemmata \ref{Vacuum Bianchi identity I} and \ref{Vacuum Bianchi identity II}]
The Bianchi identity for the induced affine connection is
\begin{align}\label{bianchi_abstract}
B_{AA'BB'CC'DD'}{}^{EE'} + B_{CC'AA'BB'DD'}{}^{EE'}  + B_{BB'CC'AA'DD'}{}^{EE'}  = 0,
\end{align}
where
\begin{align}
B_{AA'BB'CC'DD'}{}^{EE'} &=  \nabla_{AA'}\big(2\delta_{[B}{}^{E}\Phi_{C]DB'C'}\delta_{D'}{}^{E'} + \Lambda_{BC}\delta_{(B'}{}^{E'}\eta_{C')D'}\delta_{D}{}^{E} \\ &+ \eta_{B'C'}\big(\Psi_{BCD}{}^{E}-2\Lambda_{D(B}\delta_{C)}{}^{E}\big)\delta_{D'}{}^{E'} + \eta_{B'C'}\Phi_{BCD'}{}^{E'}\delta_{D}{}^{E}\big) \nonumber .   
\end{align}
Take (\ref{bianchi_abstract}),  trace on $D'$,  $E'$,  symmetrise on $A',B'$ and skew-symmetrise on $A,B$.  It yields 
\begin{align}\label{big_mess}
0 = \ &2\nabla_{CC'}\delta_{[A}{}^{E}\Phi_{B]D}{}^{A'B'} + 2\nabla_{[B|}{}^{(A'|}\delta_{C}{}^{E}\Phi_{|A]DC'}{}^{|B')} - 2\nabla_{[B}{}^{(A'|}\delta_{A]}{}^{E}\Phi_{CDC'}{}^{|B')} \\ &+2\nabla_{[B|}{}^{(A'}\delta_{C'}{}^{B')}\Psi_{|A]CD}{}^{E} + 2\nabla_{[B|}{}^{(A'}\delta_{C'}{}^{B')}\Lambda_{|A]D}\delta_{C}{}^{E} + 2\nabla_{[B|}{}^{(A'}\delta_{C'}{}^{B')}\Lambda_{CD}\delta_{|A]}{}^{E}. \nonumber
\end{align} 
Trace (\ref{big_mess}) over $C,  E$ and trace on $B',C'$:
\begin{align}
\nabla_{[A}{}^{A'}\Phi_{B]CA'B'} + 3\nabla_{[A|B'}\Lambda_{|B]C} &= 0.  \label{Bianchi1}
\end{align}
On the other hand,  trace (\ref{big_mess}) over $B,E$ and $B',C'$ to get (after clearing denominators and relabelling indices),
\begin{align} \label{Bianchi2}
\nabla_{DA'}\Psi_{ABC}{}^{D} + (1-2n)\nabla_{(A|A'}\Lambda_{|B)C} + (1-2n)\nabla_{(A|B'}\Phi_{|B)CA'}{}^{B'} = 0.  
\end{align}
Now if $\Phi_{ABA'B'} = 0$ it is easy to see (\ref{Bianchi1}) and (\ref{Bianchi2}) imply $\nabla_{AA'}\Lambda_{BC} = 0$ which is Lemma \ref{Vacuum Bianchi identity I}.   Substituting this,  together with the Einstein condition back into (\ref{big_mess}) immediately leads to Lemma \ref{Vacuum Bianchi identity II}.   
\end{proof}

\end{document}